\documentclass[10pt,twoside,reqno]{amsart}
\usepackage[foot]{amsaddr}

\usepackage{fullpage}
\usepackage{amssymb,amsfonts,amsmath,amsthm}
\usepackage{ stmaryrd}
	\theoremstyle{definition}
	\newtheorem{thm}{Theorem}
	\newtheorem*{thm*}{Theorem}

	\newtheorem{defi}[thm]{Definition} 
	
	\newtheorem{rmk}[thm]{Remark}

\numberwithin{equation}{section}
\numberwithin{thm}{subsection}

\usepackage{enumerate}
\usepackage[shortlabels]{enumitem}
\newlist{sumside}{itemize}{2}
\setlist[sumside]{topsep=0pt,label=-,leftmargin=2em,noitemsep}

\usepackage{hyperref}
\usepackage[dvipsnames]{xcolor}
\newcommand\myshade{85}

\definecolor{darkolivegreen}{rgb}{0.33, 0.42, 0.18}
\definecolor{darklavender}{rgb}{0.45, 0.31, 0.59}
\definecolor{darkraspberry}{rgb}{0.53, 0.15, 0.34}
\definecolor{darkred}{rgb}{0.55, 0.0, 0.0}

\colorlet{mylinkcolor}{darkred}
\colorlet{mycitecolor}{YellowOrange}
\colorlet{myurlcolor}{Aquamarine}
\hypersetup{
	linkcolor  = mylinkcolor,
	citecolor  = mycitecolor!\myshade!black,
	urlcolor   = myurlcolor!\myshade!black,
	colorlinks = true,
}

\usepackage{tikz}

\allowdisplaybreaks
\newcommand{\ignore}[1]{}
\usepackage{parskip}

\usepackage{array}

\usepackage{newtxtext}
\usepackage[libertine]{newtxmath}

\usepackage{etoolbox}
\patchcmd{\section}{\scshape}{\bfseries}{}{}
\makeatletter
\renewcommand{\@secnumfont}{\bfseries}
\makeatother

\makeatletter
\patchcmd{\@settitle}{\uppercasenonmath\@title}{\Large}{}{}
\patchcmd{\@setauthors}{\MakeUppercase}{\normalsize}{}{}
\makeatother

\makeatletter
\def\th@definition{%
	\thm@headfont{\bfseries}
	\normalfont 
}
\makeatother

\begin{document}
\thispagestyle{empty}

\title{Staircases to analytic sum-sides for many new integer partition identities of Rogers-Ramanujan type}

\author{Shashank Kanade\textsuperscript{1}}
\address{\textsuperscript{1} University of Denver, Denver, USA}
\thanks{S.K. was partially supported by the PIMS Post-doctoral Fellowship awarded by Pacific Institute for the Mathematical Sciences, the Endeavour Research Fellowship (2017) awarded by the Department of Education and Training, Australian Government and
by a Start-up Grant provided by University of Denver.}
\email{shashank.kanade@du.edu}

\author{Matthew C.\ Russell\textsuperscript{2}}
\address{\textsuperscript{2} Rutgers, The State University of New Jersey, Piscataway, USA}
\email{russell2@math.rutgers.edu}

\begin{abstract}
	We utilize the technique of staircases and jagged partitions to provide analytic sum-sides to some old and new partition identities
	of Rogers-Ramanujan type.
	Firstly, we conjecture a class of new partition identities related to the principally specialized characters of certain level $2$ modules for the affine Lie algebra $A_9^{(2)}$. 
	Secondly, we provide analytic sum-sides to some earlier conjectures of the authors. 
	Next, we use these analytic sum-sides to discover a number of further generalizations. 
	Lastly, we apply this technique to the well-known Capparelli 
	identities and present analytic sum-sides which we believe to be new.
	All of the new conjectures presented in this article are supported by a strong mathematical evidence.
\end{abstract}

\maketitle

\section{Introduction}
Recall the famous Rogers-Ramanujan identities \cite{And-book,Sills-book}:
\begin{thm*} For any positive integer $n$ we have:
	\begin{enumerate}[noitemsep,topsep=0pt]
		\item Number of partitions of $n$ in which adjacent parts differ by at least $2$
		is the same as the number of partitions of $n$ in which each part is $\equiv \pm 1\pmod{5}$.
		
		In generating function form, we have:
		\begin{equation}
		\sum\limits_{n\geq 0}\dfrac{q^{n^2}}{(1-q)(1-q^2)\cdots(1-q^n)} = \prod\limits_{
			\substack{m\geq 1\\m\, \equiv\, \pm 1\,\,(\mathrm{mod}\, 5)   }}\frac{1}{1-q^m}.
		\end{equation}
 	\item Number of partitions of $n$ in which adjacent parts differ by at least $2$ and $1$ does not appear as a part
		is the same as the number of partitions of $n$ in which each part is $\equiv \pm 2\pmod{5}$.
		
		In generating function form, we have:
		\begin{equation}
		\sum\limits_{n\geq 0}\dfrac{q^{n^2+n}}{(1-q)(1-q^2)\cdots(1-q^n)} = \prod\limits_{
			\substack{m\geq 1\\m\, \equiv\, \pm 2\,\,(\mathrm{mod}\, 5)}}\frac{1}{1-q^m}.
		\end{equation}
	\end{enumerate}
\end{thm*}
In the above, we interpret the $q$-series identities as identities of purely formal series.
We introduce the terms ``partition-theoretic sum-side'' to refer to the difference conditions, 
``analytic sum-side'' to refer to the sum in the $q$-series expansions and ``product-side''
to the product in the $q$-series.
The present paper deals with experimentally finding many identities of Rogers-Ramanujan type as we now explain.

\subsection{Affine Lie algebras and integer partition identities}
Affine Lie algebras have been an important source of new and intriguing integer partition identities and $q$-series identities.
We refer the reader to  \cite{Sills-book} for an excellent review. 

Building on the work \cite{FL,LM}, in a series of papers \cite{LW1,LW2,LW3,LW4} Lepowsky and Wilson showed how to interpret and prove Rogers-Ramanujan type identities using principally specialized standard modules for affine Lie algebras. 
In these works, the product-sides of the identities arise from Weyl-Kac character formula combined with Lepowsky's numerator formula and the 
partition-theoretic sum-sides arise from Lepowsky-Wilson's $Z$-algebras.
In \cite{C1}, Capparelli used Meurman and Primc's variant \cite{MP1} of Lepowsky and Wilson's method of $Z$-algebras to conjecture new partition identities using level 3 standard modules for $A^{(2)}_2$. 

In this setup, the sum-side partition conditions (generically, these are multi-color partitions) follow by reducing a given Poincar\'e-Birkhoff-Witt-type spanning set to a basis --- a reduction that is achieved by using vertex-algebraic ``relations.''
One can therefore be assured that representations of any affine Lie algebra at any positive integral level lead to some sort of identities involving partitions, with the caveat that such identities are generically extremely complicated. Nonetheless, the point is that the characters of standard modules of affine Lie algebras
are a treasure trove of many interesting and as yet unknown integer partition identities.

In principle, the process of conjecturing partition identities using $Z$-algebras (or other vertex-algebraic methods) for any given affine Lie algebra at any given (positive integral) level could be utilized to discover new identities, however this gets notoriously tedious when the rank
of the algebra and/or the level of the module become large. This necessitates the need for new techniques of investigation.

In \cite{KR}, we initiated a study to discover new partition identities using experimental methods, and we discovered six new conjectural identities. Three of these six turned out to be related to principally specialized characters of the level $3$ standard modules for the affine Lie algebra $D_4^{(3)}$. In a current work in progress \cite{KNR}, we have initiated a search for identities that mimic certain $Z$-algebraic considerations.
However, in these works, a priori we were not specifically looking for identities related to affine Lie algebras. 
After the search, one had to check if the newly found identities would be the characters of standard modules for some affine Lie algebra.

In the present paper, we undertake the exploration in a fundamentally different philosophical direction. 
We start with a specific algebra at a specific level (the affine Lie algebra $A_9^{(2)}$ at level 2) and we start with the principally specialized characters of the vacuum spaces (with respect to the principal Heiseberg algebra) of the corresponding standard modules. These characters naturally factor into an infinite product due to Weyl-Kac character formula and Lepowsky's numerator formula. 
We then use experimental methods to conjecture corresponding partition-theoretic sum-sides.
For the algebras $A_{2n+1}^{(2)}$, the products arising from those level $2$ modules that are contained in the tensor product of two inequivalent level $1$ modules \cite{BM-crystal} seem promising from the partition-theoretic viewpoint: 
among other things, their inverse Euler transform is periodic with only entries being $0$ or $\pm 1$. By a slight abuse of terminology, we say that the inverse Euler transform of the $q$-series $\prod_{m\geq 1}(1-q^{m})^{-a_m}$ is the sequence $\{a_m\}_{m\geq 1}$.
The following table summarizes the known information and explains why $A_9^{(2)}$ was a natural candidate to explore (see \cite{B} for the corresponding product-sides):
\begin{center}
	\begin{small}
		{
			\def\arraystretch{1.4}
			\begin{tabular}{|c|c|c|}
				\hline 
				\textit{Algebra} & \textit{Product Sides/Identities} & \textit{Sum-side status} \\ 
				\hline \hline
				$A_3^{(2)}$ &  {Alladi's companion to Schur's identity \cite{And-310}} & 
				$Z$-algebraic interpretation given \cite{T} \\ 
				\hline 
				$A_5^{(2)}$ &  G\"ollnitz-Gordon identities & 
				$Z$-algebraic interpretation given \cite{K-GGZ} \\ 
				\hline 
				$A_7^{(2)}$ &  Rogers-Ramanujan identities & 
				$Z$-algebraic interpretation given in \cite{BM-RR} \\ 
				\hline 
				$A_9^{(2)}$
				&  Present article
				&  $Z$-algebraic interpretation is a future work \\ 
				\hline 
				$A_{11}^{(2)}$ & Nandi's products \cite{N} & 
				Sum-sides given in \cite{N} using level 4 modules for $A_2^{(2)}$ \\ 
				\hline 
			\end{tabular}
		}
	\end{small}
\end{center}
To the best of our knowledge, for $A_{2n+1}^{(2)}$ with $n\geq 6$, the products do not correspond to known partition identities.

\subsection{Staircases to the sum-sides}

A striking feature of the identities found using $Z$-algebras is that the sum-sides of such identities are inherently partition-theoretic. Generically, it appears to be a hard task to find ``nice'' analytic sum-sides that count the sum-side partitions. In a few cases, certain candidates for analytic sum-sides are known using Bailey techniques (see for instance \cite{Sills-nandi}), but it is far from obvious how these analytic sum-sides are related to the partition theoretic sum-sides.

For the identities presented in this article, we are able to provide analytic companions to our partition-theoretic sum-sides, by using staircases and jagged partitions. Given a partition $\pi: \lambda_1+\cdots+\lambda_j$ written in a weakly increasing order and a positive integer $s$, the (ordered) sequence $\mu: \lambda_1, \lambda_2-s,\lambda_3-2s,\cdots,  \lambda_{j}-(j-1)s$ is said to be obtained from $\pi$ by removing an $s$-staircase. As such, $\mu$ may not be in a weakly increasing order and
may also have non-positive entries. Such a 	sequence is called a jagged partition. 
For all the identities in this paper, removing an appropriate staircase leads to interesting jagged partitions whose generating functions could be written down explicitly. One then replaces the staircases to arrive at analytical sum-sides for the original identities.

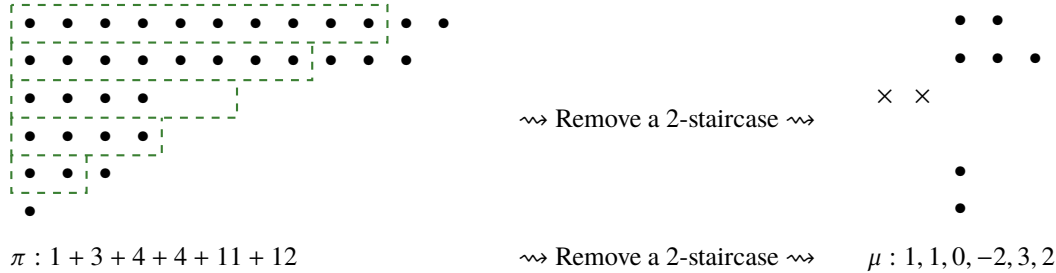
\begin{alignat*}{3}
&\begin{matrix}
\begin{tikzpicture}
\foreach \x in {0,...,11}
\draw (\x*.5, 0) node {$\bullet$};
\foreach \x in {0,...,10}
\draw (\x*.5, -0.5) node {$\bullet$};
\foreach \x in {0,...,3}
\draw (\x*.5, -1) node {$\bullet$};
\foreach \x in {0,...,3}
\draw (\x*.5, -1.5) node {$\bullet$};
\foreach \x in {0,...,2}
\draw (\x*.5, -2) node {$\bullet$};
\foreach \x in {0,...,0}
\draw (\x*.5, -2.5) node {$\bullet$};
\draw[dashed, thick, OliveGreen] (-0.25,0.25) -- (-0.25,-0.25) -- (4.75,-0.25) -- (4.75,0.25) --  (-0.25,0.25);
\draw[dashed, thick,  OliveGreen] (-0.25,-0.25) -- (-0.25,-0.75) -- (3.75,-0.75) -- (3.75,-0.25);
\draw[dashed, thick,  OliveGreen] (-0.25,-0.75) -- (-0.25,-1.25) --  (2.75,-1.25) -- (2.75,-0.75);
\draw[dashed, thick, OliveGreen] (-0.25,-1.25) -- (-0.25,-1.75) -- (1.75,-1.75) -- (1.75,-1.25);
\draw[dashed, thick, OliveGreen] (-0.25,-1.75) -- (-0.25,-2.25) -- (0.75,-2.25) -- (0.75,-1.75);
\end{tikzpicture}
\end{matrix}
&& \qquad \rightsquigarrow {\text{Remove a 2-staircase}} \rightsquigarrow\qquad
&&
\begin{matrix}
\begin{tikzpicture}
\foreach \x in {0,...,1}
\draw (\x*.5, 0) node {$\bullet$};
\foreach \x in {0,...,2}
\draw (\x*.5, -0.5) node {$\bullet$};
\foreach \x in {-2,...,-1}
\draw (\x*.5, -1) node {$\times$};
\foreach \x in {0,...,0}
\draw (\x*.5, -2) node {$\bullet$};
\foreach \x in {0,...,0}
\draw (\x*.5, -2.5) node {$\bullet$};
\end{tikzpicture}
\end{matrix}\\
&\pi: 1 + 3 + 4 + 4 + 11 + 12
&&\qquad\rightsquigarrow {\text{Remove a 2-staircase}} \rightsquigarrow\qquad
&&\mu: 1, 1, 0, -2, 3, 2
\end{alignat*}

The identities in $A_9^{(2)}$ lead to analytic sum-sides which differ only in the linear term in the exponent of $q$ in each summand. Varying this linear term further leads us to six more conjectural identities. 
We then use this technique to provide analytic sum sides to certain previous conjectures, namely,
Identities $I_5$ and $I_6$ from  \cite{KR} and Identities $I_{5a}$ and $I_{6a}$ from \cite{R}.
Again, variations on the analytic sum-sides for \cite[$I_6$]{KR} lead us to three more conjectural identities.
One of these three identities has an asymmetric product-side; we present 
one further identity whose product-side has negatives of the residues from the aforementioned asymmetric product.
For every new (conjectural) identity presented in this paper, we are able to provide both the partition-theoretic and analytic sum-sides and prove that the analytic sum-sides are indeed generating functions of the partition-theoretic sum-sides. 
Lastly, we also discuss analytic sum-sides to the two Capparelli identities.

As mentioned above, jagged partitions play a crucial role in the present paper. They first arose in the physics literature
\cite{FJM1} in the analysis of fermionic characters for certain superconformal minimal models.
Moreover, in \cite{FJM2,FJM3} some beautiful identities for jagged partitions were established. 
In \cite{L}, Lovejoy established certain constant term identities related to generating functions for jagged partitions.
In \cite{ABM}, new analytic sum-sides for Schur's identity were found using staircases.
Very recently, in \cite{DL}, generalizations of the (first) Capparelli identity were found by utilizing jagged (over)partitions.
Our treatment of the first Capparelli identity is precisely a ``dilated'' version of the argument in \cite{DL}.
Notably, in \cite{Cap-higher}, Capparelli used staircases in his investigation of the identities related to the standard modules at levels $5$ and $7$ for the affine Lie algebra $A_2^{(2)}$.

\subsection{Verification} 
All the new conjectural identities in this paper have been verified up to the coefficient of $q^{500}$. 
Unlike \cite{KR} and \cite{R} where recursions based on partition-theoretic sum-sides were used for such a verification,
here we directly use the analytic sum-sides. 
Maple code for verification can be found appended in the plain-text format  at the end of the \texttt{.tex} file of this paper on arXiv.

\subsection{Future work and work in progress}

Proving the $q$-series identities in this paper is a work in progress. 

We expect a vertex-operator-theoretic interpretation of the identities 
arising from $A_9^{(2)}$ to be tedious. It is quite possible that such an investigation
may actually lead to completely different partition-theoretic sum-sides than the ones 
conjectured here.

Investigation of identities related to analogous level 2 standard modules for all $A_{\text{odd}}^{(2)}$ 
and higher level standard modules for $A_2^{(2)}$ would be extremely interesting; see \cite{Cap-higher}, \cite{MS} and \cite[Section 6.4]{Sills-book}. 
Nandi's identities (originating from level 4 standard modules for $A_2^{(2)}$) \cite{N} still remain quite difficult;
see \cite{Sills-nandi} for some recent results. These identities merit a closer study in the light of techniques presented
here.

We are working on experimentally finding more $q$-series identities of the kind presented here using jagged partitions
and staircases.

\subsection*{Acknowledgments} It is our pleasure to thank George E.\ Andrews, James Lepowsky, Mirko Primc, Andrew V.\ Sills and Doron Zeilberger
for their interest in our work and encouragement. We also thank Jeremy Lovejoy for a correspondence regarding \cite{DL}.

\section{Preliminaries}

We use the standard conventions regarding the $q$-Pochhammer symbols:
\begin{align}
(a;q)_n &= \prod_{1\leq t < n} (1-aq^t),\\
(a;q)_\infty &= \prod_{1\leq t } (1-aq^t),\\
(a_1,a_2,\dots,a_j;q)_m&=(a_1;q)_m(a_2;q)_m\cdots(a_j;q)_m.
\end{align}
For us, all identities presented in this paper are formal power series identities, and we shall expand expressions such as $\frac{1}{1-q}$ using geometric series.
We shall frequently use the following fundamental $q$-series identities due to Euler:
\begin{align}
\left(x;q\right)_\infty^{-1}&=
\sum_{n\geq 0}\dfrac{x^n}{(q;q)_n}\label{eqn:eulerp},\\
\left(-x;q\right)_\infty&=
\sum_{n\geq 0}\dfrac{x^nq^{n(n-1)/2}}{(q;q)_n}\label{eqn:eulerd}.
\end{align}

We write partitions of a positive integer in a \emph{weakly increasing} order. Sub-partitions of a partition $\pi$ will always refer
to \emph{contiguous} portions of $\pi$.

Suppose $\pi: p_1 + p_2 + \dots + p_n$ is a partition, where sometimes we may have to let $p_1=0$. Let $s$ be a positive integer. 
Let $\mu$ be the sequence $p_1, p_2-s,p_3-2s,\dots,p_n-(n-1)s$.
We say that $\mu$ is obtained by deleting an $s$-staircase from $\pi$. Note that $\mu$ may have non-positive entries and may not be in a weakly increasing order (which is the reason why we separate entries of $\mu$ by a comma rather than a plus sign), however,
successive differences in $\mu$ are at least $-s$.
We call such $\mu$ jagged partitions.

Suppose that the generating function of a set $\mathcal{P}$ of partitions is $\sum\limits_{m,n\geq 0}a_{m,n}x^mq^n$, i.e., there are
$a_{m,n}$ partitions of $n$ of length $m$ in $\mathcal{P}$. Then, 
the generating function of jagged partitions obtained by removing an $s$-staircase from each of the partitions of $\mathcal{P}$ 
is $\sum\limits_{m,n\geq 0}a_{m,n}x^mq^{-sm(m-1)/2}q^n$.

Jagged partitions emerging in this paper will have a very specific structure. We shall scan a jagged partition $\mu$ from left to right  and identify maximum jagged portions of $\mu$ corresponding to each of the positive integers $1,2,\dots$ (sometimes $0$ will have to be considered as well). If $\mu$ is obtained by removing an $s$-staircase, the block
corresponding to $j$ is defined as the maximal contiguous block of $\mu$
starting with $j$ and containing only integers from $\{j,j-1,\dots,j-s\}$; either of these conditions may be sometimes slightly relaxed but it will be clear what we mean.
We shall designate such blocks using regular expressions. 
If $\mathbf{P}$ is a pattern of integers (for example $\mathbf{P}=``2,1"$), then:
\begin{enumerate}
\item[$\mathbf{P}^*$] corresponds to a string of either $0$ or more contiguous blocks of $\mathbf{P}$.
\item[$\mathbf{P}^+$] will be a string of 
$1$ or more contiguous blocks of $\mathbf{P}$.
\item[$\mathbf{P}^\bullet$] will be either the empty string or $\mathbf{P}$ itself.
\end{enumerate}
For instance, if the jagged partition 
$2,1,2,1,3,3,4,3$ 
is obtained by removing a $1$-staircase from some partition $\pi$, its maximal block corresponding to $2$ matches $[2,1]^*$ (and also $[2,1]^+$ but not $[2,1]^\bullet$), the maximal block
corresponding to $3$ matches $3^*$  (and also $3^+$ but not
$3^\bullet$), the one for $4$ matches $[4,3]^\bullet$ etc.

We shall always denote a partition by $\pi$ and the jagged partition obtained by removing an $s$-staircase by $\mu$.

\section{Identities related to certain level $2$ modules for $A_9^{(2)}$}

\subsection{The symmetric conjectures}

We use the standard convention \cite{Kac} for designating the nodes in the Dynkin diagram of $A_9^{(2)}$,
 see Figure \ref{fig:A92}.
In our conjectures, the product sides are precisely the principally specialized characters of the vacuum spaces 
(with respect to the principal Heisenberg subalgebra) of certain level 2 standard modules as we specify below. These modules are contained in the tensor product of two inequivalent level $1$ modules \cite{BM-crystal}. 
See \cite{LW1,LW2,LW3,LW4} and \cite{B} for the relevant terminology and the product-sides.
\begin{figure}\label{fig:A92}
	\begin{tikzpicture}
	\draw [-] (1,1) -- (2.5,2);
	\draw [-] (1,3) -- (2.5,2);
	\draw [-] (2.5,2) -- (4,2);
	\path (4,2) -- node {\(==\Leftarrow =\)} (5.5,2);
	\draw [fill] (1,1) circle [radius=0.1];
	\draw [fill] (1,3) circle [radius=0.1];
	\draw [fill] (2.5,2) circle [radius=0.1];
	\draw [fill] (4,2) circle [radius=0.1];
	\draw [fill] (5.5,2) circle [radius=0.1];
	\node [below] at (1,0.9) {$\alpha_1$};
	\node [below] at (1,2.9) {$\alpha_0$};
	\node [below] at (2.5,1.9) {$\alpha_2$};
	\node [below] at (4,1.9) {$\alpha_3$};
	\node [below] at (5.5,1.9) {$\alpha_4$};
	\node [above] at (1,1.1) {$1$};
	\node [above] at (1,3.1) {$1$};
	\node [above] at (2.5,1.9+0.2) {$2$};
	\node [above] at (4,1.9+0.2) {$2$};
	\node [above] at (5.5,1.9+0.2) {$2$};
	\end{tikzpicture}
	\caption{Dynkin diagram of $A_9^{(2)}$. Labels $\alpha_\bullet$ enumerate the nodes. Numerical labels are coefficients for the canonical central element.}
\end{figure}
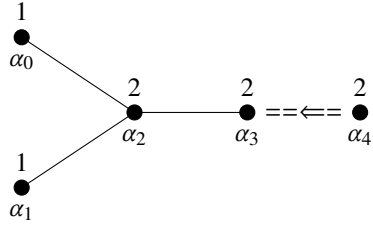
Ideas stemming from \cite{KNR, KR, N, R} suggest the following conjectures.
These conjectures have been checked up to the partitions of 
$N=500$.

The following sum-side conditions are common to the three ensuing conjectures:
\begin{enumerate}
	\item No consecutive parts allowed.
	\item Odd parts do not repeat.
	\item Even parts appear at most twice.
	\item If a part $2j$ appears twice then $2j\pm 3,2j\pm 2$
	(and an additional copy of $2j$, but this is subsumed in the third condition, also 
	$2j\pm 1$, but this is subsumed in the first condition) are forbidden to appear at all.
\end{enumerate}	

Equivalently:
\begin{enumerate}
	\item No consecutive parts allowed.
	\item Odd parts do not repeat.
	\item For a contiguous sub-partition $\lambda_i+\lambda_{i+1}+\lambda_{i+2}$, we have
	$|\lambda_i - \lambda_{i+2}| \geq 4$ if
	$\lambda_{i+1}$ is even and appears more than once.
\end{enumerate}	

\subsubsection{Identity 1:}

Arises from the module $L(\Lambda_0 + \Lambda_1)$ 
of $A_9^{(2)}$.

Product:  
$$\frac 1 {\left(q,q^4,q^6,q^8,q^{11};q^{12}\right)_\infty}.$$

Partition-theoretic sum-side: Above conditions, along with the initial condition that $2+2$ is not allowed as a sub-partition.

Example: There are ten partitions of each type for $n=12$.

\begin{minipage}[t]{3.7in}
	Product side: \\
	$1+11$ \\ 
	$4+8$ \\ 
	$1+1+1+1+8$ \\ 
	$6+6$ \\ 
	$1+1+4+6$ \\ 
	$1+1+1+1+1+1+6$ \\ 
	$4+4+4$ \\ 
	$1+1+1+1+4+4$ \\ 
	$1+1+1+1+1+1+1+1+4$ \\ 
	$1+1+1+1+1+1+1+1+1+1+1+1$
\end{minipage}
\hfill
\begin{minipage}[t]{2.3in}
Partition-theoretic sum-side: \\
	$12$ \\  
	$1+11$ \\ 
	$2+10$ \\ 
	$3+9$ \\ 
	$4+8$ \\ 
	$1+3+8$ \\ 
	$5+7$ \\ 
	$1+4+7$ \\ 
	$6+6$ \\ 
	$2+4+6$
\end{minipage}

\subsubsection{Identity 2:}

Arises from the module $L(\Lambda_3)$ 
of $A_9^{(2)}$.

Product:  
\vspace{-.2in}

$$\frac {\left(q^6;q^{12}\right)_\infty}{\left(q^2,q^3,q^4,q^8,q^9,q^{10};q^{12}\right)_\infty}
= \frac {\left(-q^3;q^6\right)_\infty\left(q^6;q^6\right)_\infty}{\left(q^2;q^2\right)_\infty}  $$

Partition-theoretic sum-side: Above conditions, along with the initial condition that $1$ is forbidden to appear.

\subsubsection{Identity 3:}

Arises from the module $L(\Lambda_5)$ 
of $A_9^{(2)}$.

Product:  
$$\frac 1 {\left(q^4,q^5,q^6,q^7,q^8;q^{12}\right)_\infty}.$$

Partition-theoretic sum-side: Above conditions, along with the initial condition that $1$, $2$, and $3$ are all forbidden as parts.

\subsection{Analytic sum-sides for the symmetric conjectures}

\subsubsection{Identity 1}

Let $\pi$ be a partition counted in the sum-side and remove a $2$-staircase to obtain a jagged partition $\mu$.

Looking at the restrictions on $\pi$ it is clear the corresponding restrictions on $\mu$ amount to forbidding the appearance of the following blocks.
\begin{enumerate}
\item $j, j-1$. \label{it:id1noconsec}
\item $j, j-2$ if $j$ is odd\label{it:id1nooddrepeat}. 
\item $j, j-2, j-4$ if $j$ is even.
\item $j, j-2, j-2$ if $j$ is even.
\item $j, j, j-2$ if $j$ is even.
\item $j, j+1, j-1$ if $j$ is odd.
\item $2,0$ is not allowed to appear in $\mu$.
\end{enumerate}

This implies that if $j$ is odd, the maximal block corresponding to $j$ in $\mu$ is of the form $j^*$.
Similarly, if $j\neq 2$ is even, the maximal block corresponding to $j$ in $\mu$ is of the form $j^*,[j-2, j]^*$. 
Also, no block of the shape $j-1,j+1,j$ appears for $j$ odd.
It is now straightforward to obtain that the $(x,q)$-generating function for such $\mu$ is
given by:
\begin{align}
&\prod_{i\, \mathrm{odd}}\dfrac{1}{1-xq^i}
\prod_{j\,\mathrm{odd}}(1-xq^{j-1}\cdot xq^{j+1}\cdot xq^{j})
\prod_{k\,\mathrm{even}, k\geq 4}\dfrac{1}{1-xq^k\cdot xq^{k-2}}
\prod_{l\,\mathrm{even}}\dfrac{1}{1-xq^l}\nonumber\\
&
=\left(xq;q\right)^{-1}_\infty\left(x^2q^6;q^4\right)^{-1}_\infty\left(x^3q^9;q^6\right)_\infty\nonumber\\
&= \left(\sum_{i\ge 0} \frac{x^i q^i}{\left(q;q\right)_i}\right)
\left(\sum_{j\ge 0} \frac{x^{2j}q^{6j}}{\left(q^4;q^4\right)_j}\right)
\left(\sum_{k \ge 0} \frac{\left(-1\right)^kx^{3k}q^{3k^2+6k}} {\left(q^6;q^6\right)_k}\right) \nonumber\\
&= \sum_{i,j,k\ge 0} \left(-1\right)^k \frac{x^{i+2j+3k} q^{i+6j+3k^2+6k}}{\left(q;q\right)_i\left(q^4;q^4\right)_j\left(q^6;q^6\right)_k}.
\end{align}
Now we reinstate the $2$-staircase, i.e., we let $x^m\mapsto x^mq^{m(m-1)}$ to get that the required analytic sum side is:
\begin{align}
 \sum_{i,j,k\ge 0} \left(-1\right)^k \frac{x^{i+2j+3k} q^{(i+2j+3k)(i+2j+3k-1) + i+6j+3k^2+6k}}{\left(q;q\right)_i\left(q^4;q^4\right)_j\left(q^6;q^6\right)_k}.
\end{align}

One now takes $x\mapsto 1$ to deduce the conjectures.
We shall omit this last step in the identities below.
\subsubsection{Identity 2} For this identity, only the initial conditions change. We get that $\mu$ must avoid blocks corresponding to $j=1$.
We have that the generating function for $\mu$ is:
\begin{align}
&\prod_{i\, \mathrm{odd}, i\geq 3}\dfrac{1}{1-xq^i}
\prod_{j\,\mathrm{odd}, j\geq 3}(1-xq^j\cdot xq^{j+1}\cdot xq^{j-1})
\prod_{k\,\mathrm{even} }\dfrac{1}{1-xq^k\cdot xq^{k-2}}
\prod_{l\,\mathrm{even}}\dfrac{1}{1-xq^l}\nonumber\\
&=\left(xq^2;q\right)^{-1}_\infty\left(x^3q^9;q^6\right)_\infty \left(x^2q^2;q^4\right)^{-1}_\infty\nonumber\\
&= \left(\sum_{i\ge 0} \frac{x^i q^{2i}}{\left(q;q\right)_i} \right)
\left(\sum_{k \ge 0} \frac{\left(-1\right)^kx^{3k}q^{3k^2+6k}}{\left(q^6;q^6\right)_k}\right) \left(\sum_{j\ge 0} \frac{x^{2j}q^{2j}}{\left(q^4;q^4\right)_j}\right) \nonumber\\
&= \sum_{i,j,k\ge 0} \frac{\left(-1\right)^k x^{i+2j+3k} q^{2i+2j+3k^2+6k}}{\left(q;q\right)_i\left(q^4;q^4\right)_j\left(q^6;q^6\right)_k}. 
\end{align}
Reinstating the $2$-staircase, we obtain:
\begin{align}
\sum_{i,j,k\ge 0} \frac{\left(-1\right)^k x^{i+2j+3k} q^{(i+2j+3k)(i+2j+3k-1)+2i+2j+3k^2+6k}}{\left(q;q\right)_i\left(q^4;q^4\right)_j\left(q^6;q^6\right)_k}. 
\end{align}

\subsubsection{Identity 3} 
For this identity, we omit the blocks corresponding to parts $1$, $2$ and $3$ from the generating function for $\mu$:
\begin{align}
&\prod_{i\, \mathrm{odd}, i\geq 5}\dfrac{1}{1-xq^i}
\prod_{j\,\mathrm{odd}, j\geq 5}(1-xq^j\cdot xq^{j+1}\cdot xq^{j-1})
\prod_{k\,\mathrm{even}, k\geq 4 }\dfrac{1}{1-xq^k\cdot xq^{k-2}}
\prod_{l\,\mathrm{even}, k\geq 4}\dfrac{1}{1-xq^l}\nonumber\\
&=\left(xq^4;q\right)^{-1}_\infty\left(x^3q^{15};q^6\right)_\infty\left(x^2q^6;q^4\right)^{-1}_\infty\nonumber\\
&=\left(\sum_{i\ge 0} \frac{x^i q^{4i}}{\left(q;q\right)_i}\right)\left(\sum_{k \ge 0} \frac{\left(-1\right)^kx^{3k}q^{3k^2+12k}}{\left(q^6;q^6\right)_k}\right)\left(\sum_{j\ge 0} \frac{x^{2j}q^{6j}}{\left(q^4;q^4\right)_j}\right) \nonumber\\
&=\sum_{i,j,k\ge 0} \frac{\left(-1\right)^k x^{i+2j+3k} q^{4i+6j+3k^2+12k}}{\left(q;q\right)_i\left(q^4;q^4\right)_j\left(q^6;q^6\right)_k}.
\end{align}
Reinstating the $2$-staircase, we obtain:
\begin{align}
\sum_{i,j,k\ge 0} \frac{\left(-1\right)^k x^{i+2j+3k} q^{(i+2j+3k)(i+2j+3k-1)+4i+6j+3k^2+12k}}{\left(q;q\right)_i\left(q^4;q^4\right)_j\left(q^6;q^6\right)_k}.
\end{align}

\subsection{An intriguing relation}

We now deduce a relation that holds among the symmetric $A_9^{(2)}$ conjectures.
Let us denote the sum-side in the identity $i=1,2,3$ by $S_i(x,q)$ where $x$ counts number of parts
and $q$ corresponds to the number being partitioned.

\begin{thm}\label{thm:intrigueSum}
	We have that:
	\begin{align}
	S_1(x,q)=S_2(x,q)+xqS_3(x,q).
	\end{align}
\end{thm}
\begin{proof}
	Recall the initial conditions: 
	\begin{itemize}
		\item[$S_1$:] $2+2$ is forbidden,
		\item[$S_2$:] $1$ is forbidden,
		\item[$S_3$:] $1,2,3$ are forbidden.
	\end{itemize}
	
	Let $\lambda: \lambda_1+\lambda_2+\dots+\lambda_k$ be a partition
	of $n$ counted in $S_1$ (recall that we have been using a \emph{weakly increasing} order)
	and consider the maximal (possibly empty) contiguous string of odds starting with $1$ contained in $\lambda$.
	Let us denote this string by $\sigma: \lambda_1=1 + \dots +\lambda_s$
	(note that $\sigma$ may be an empty string, in which case we let $s=0$).
	Also note that parts in $\sigma$ are strictly increasing, since odds are not allowed to repeat.
	
	Now, depending on the parity of $s$, we transform the $\sigma$ substring of $\lambda$ (keeping the rest of $\lambda$ unchanged) to obtain new partitions.
	\begin{enumerate}
		\item[$s$ even:] Replace every pair $\lambda_{2i-1} + \lambda_{2i}$ of adjacent parts appearing in $\sigma$ by their average, i.e., 
		$\frac{(\lambda_{2i-1} + \lambda_{2i})}{2} + \frac{(\lambda_{2i-1} + \lambda_{2i})}{2}$.
		
		Example: $1+3+5+7 \rightsquigarrow 2+2 + 6+6$.
		
		It is easy to check that the new partition thus obtained
		has the same length and weight as $\lambda$ and
		is counted in $S_2$.
		
		\item[$s$ odd:] Replace every pair $\lambda_{2i} + \lambda_{2i+1}$ of adjacent parts appearing in $\sigma$ by their average, i.e., 
		$\frac{(\lambda_{2i} + \lambda_{2i+1})}{2} + \frac{(\lambda_{2i} + \lambda_{2i+1})}{2}$
		and then delete the initial $1$.
		
		Example: $1+3+5+7+9 \rightsquigarrow 4+4 + 8+8$.
		
		It is clear that the new partition obtained has weight and length one lower than that of $\lambda$ and is counted in $S_3$.
	\end{enumerate}
\end{proof}
The product-side analogue of Theorem \ref{thm:intrigueSum} (with $x\mapsto 1$) 
can be found in \cite{CH}:
\begin{thm}\label{thm:intrigueProd}
	$$\frac 1 {\left(q,q^4,q^6,q^8,q^{11};q^{12}\right)_\infty}
	=
	\frac {\left(q^6;q^{12}\right)_\infty}{\left(q^2,q^3,q^4,q^8,q^9,q^{10};q^{12}\right)_\infty}
	+ q\cdot
	\frac 1 {\left(q^4,q^5,q^6,q^7,q^8;q^{12}\right)_\infty}.$$
\end{thm}
\begin{proof}
	Follows  from Equation (12) of \cite{CH}.
\end{proof}

\section{Asymmetric companions of the $A_9^{(2)}$ conjectures}

Observe that the analytic sum-sides presented above differ only in the linear terms in the exponents of $q$. A computer search for other possible linear terms reveals further conjectures.

\subsection{Analytic forms}
\begin{align}
\sum_{i,j,k\ge 0}  \frac{\left(-1\right)^kq^{i^2+4ij+6ik+4j^2+12jk+12k^2+j}}{\left(q;q\right)_i\left(q^4;q^4\right)_j\left(q^6;q^6\right)_k} &= \frac 1 {\left(q,q^4,q^5,q^9,q^{11};q^{12}\right)_\infty}  \label{conj:new1} \\
\sum_{i,j,k\ge 0}  \frac{\left(-1\right)^kq^{i^2+4ij+6ik+4j^2+12jk+12k^2+
		i-3j}}{\left(q;q\right)_i\left(q^4;q^4\right)_j\left(q^6;q^6\right)_k} &= 
\frac 1 {\left(q,q^5,q^7,q^8,q^{9};q^{12}\right)_\infty}  \label{conj:new1a} \\
\sum_{i,j,k\ge 0} \frac{\left(-1\right)^k q^{i^2+4ij+6ik+4j^2+12jk+12k^2-2j-3k}}{\left(q;q\right)_i\left(q^4;q^4\right)_j\left(q^6;q^6\right)_k} & = 
\frac {\left(q^3;q^{12}\right)_\infty}{\left(q,q^2,q^5,q^6,q^{9},q^{10};q^{12}\right)_\infty} \label{conj:new2} \\
\sum_{i,j,k\ge 0} \frac{\left(-1\right)^k q^{i^2+4ij+6ik+4j^2+12jk+12k^2+i+2j+3k}}{\left(q;q\right)_i\left(q^4;q^4\right)_j\left(q^6;q^6\right)_k} & = \frac {\left(q^9;q^{12}\right)_\infty}{\left(q^2,q^3,q^6,q^7,q^{10},q^{11};q^{12}\right)_\infty} \label{conj:new2a} \\
\sum_{i,j,k\ge 0} \frac{\left(-1\right)^k q^{i^2+4ij+6ik+4j^2+12jk+12k^2
		-j}}{\left(q;q\right)_i\left(q^4;q^4\right)_j\left(q^6;q^6\right)_k} & = 
\frac 1 {\left(q,q^3,q^7,q^8,q^{11};q^{12}\right)_\infty}\label{conj:new3} \\
\sum_{i,j,k\ge 0} \frac{\left(-1\right)^k q^{i^2+4ij+6ik+4j^2+12jk+12k^2+2i+3j+6k}}{\left(q;q\right)_i\left(q^4;q^4\right)_j\left(q^6;q^6\right)_k} & = \frac 1 {\left(q^3,q^4,q^5,q^7,q^{11};q^{12}\right)_\infty}\label{conj:new3a}
\end{align}
Note the way in which the product sides come in pairs: \eqref{conj:new1} and \eqref{conj:new3}, 
\eqref{conj:new1a} and \eqref{conj:new3a}, and
\eqref{conj:new2} and \eqref{conj:new2a}. (The allowable congruence classes for \eqref{conj:new1} are $1,4,5,9,11 \pmod{12}$, while the allowable congruence classes for \eqref{conj:new3} are $-1,-4,-5,-9,-11 \pmod{12}$.)

\subsection{Partition-theoretic sum-sides}
\subsubsection{Identities 4 and 4a: \eqref{conj:new1} and \eqref{conj:new1a} }
The sum side of Conjecture \eqref{conj:new1} has the following difference conditions:
\begin{enumerate}
	\item No part repeats.
	\item Adjacent parts do not differ by 1 if the larger part is even.
	\item $(2j) + (2j+1) + (2j+3)$ forbidden as a sub-partition.
	\item $(2j) + (2j+2) + (2j+3)$ forbidden as a sub-partition.
	\item $(2j) + (2j+2) + (2j+4)$ forbidden as a sub-partition.
\end{enumerate}
The difference conditions for Conjecture \eqref{conj:new1a} are the same, except for an additional initial condition (hence our 4/4a terminology):
\begin{enumerate}[resume]
	\item None of $1+3$, $2+3$, $2+4$ can appear as sub-partitions. Alternately, assume that the partition starts with a fictitious $0$, and then these initial conditions are implied by the remaining difference conditions.
\end{enumerate}
We prove that these are the correct partition-theoretic sum-sides by using staircases.
Let $\pi$ be a partition counted by the sum-side for Identity 4.
Delete a $2$-staircase to obtain a jagged partition $\mu$.
It is clear that the conditions on $\pi$ amount to forbidding the following bocks in $\mu$.
\begin{enumerate}
\item $j,j-2$
\item $j,j-1$ if $j$ is odd.
\item $j,j-1,j-1$ if $j$ is even.
\item $j,j,j-1$ if $j$ is even.
\item $j,j,j$ if $j$ is even.
\end{enumerate}

Therefore, the maximal block in $\mu$ corresponding to an even value of $j$ is $[j, j-1,]^*j^\bullet, j^\bullet $ whereas the maximal block in $\mu$ corresponding to an odd $j$ is $j^*$.
We get that the generating function for $\mu$ is:
\begin{align}
&\prod_{j\,\mathrm{even}}\dfrac{1}{1-xq^j\cdot xq^{j-1}}
\prod_{k\,\mathrm{even}}(1+xq^k + xq^k\cdot xq^k)
\prod_{l\,\mathrm{odd}}\dfrac{1}{1-xq^l}\\
&=\left(xq;q\right)_{\infty}^{-1}\left(x^2q^3;q^4\right)_\infty^{-1}\left(x^3q^6;q^6\right)_\infty
 \nonumber \\
 &=\left(\sum_{i\geq 0}\dfrac{x^iq^{i}}{(q;q)_i}\right)
 \left(\sum_{j\geq 0}\dfrac{x^{2j}q^{3j}}{(q^4;q^4)_j}\right)
  \left(\sum_{k\geq 0}(-1)^k\dfrac{x^{3k}q^{3k^2+3k}}{(q^6;q^6)_k}\right)\nonumber\\
 &=\sum_{i,j,k\geq 0}(-1)^k\dfrac{x^{i+2j+3k}q^{i+3j+3k^2+3k}}{(q;q)_i(q^4;q^4)_j(q^6;q^6)_k}. 
\end{align}
Letting $x^m\mapsto x^mq^{m(m-1)}$ we get:
\begin{align}
\sum_{i,j,k\geq 0}(-1)^k\dfrac{x^{i+2j+3k}q^{
	(i+2j+3k)(i+2j+3k-1) + i + 3j+ 3k^2 + 3k}
	}{(q;q)_i(q^4;q^4)_j(q^6;q^6)_k}. 
\end{align}

The analysis for Identity $4a$ is perhaps the most delicate of all identities considered in this article.
Let $\pi$ be a partition counted in the sum-side of Identity $4a$
and let $\tilde{\mu}$ be obtained by removing a $2$-staircase from $0+\pi$.
 We make two cases and arrive at a jagged partition $\mu$ accordingly:
\begin{enumerate}[label={\alph*.}]
	\item If every occurrence of $0$ in $\tilde{\mu}$ is immediately succeeded  by $-1$ and no occurrence of $-1$ is immediately succeeded by $1$, we let $\mu=\tilde{\mu}$.
	\item In all other cases, we let $\mu$ to be obtained by simply deleting a $2$-staircase from $\pi$.
\end{enumerate}

To explain the cases, we consider the following five examples:
\begin{align}
\pi= 1+6+7 \rightsquigarrow 0 + \pi = 0 + 1+6+7
\rightsquigarrow \tilde{\mu} = 0 , -1,2,1 \rightsquigarrow^{\text{case\,\,a}}
\mu&=0,-1,2,1.
\\
\pi= 1+5+8 \rightsquigarrow 0 + \pi = 0 + 1+5+8
\rightsquigarrow \tilde{\mu} = 0 , -1,1,2 \rightsquigarrow^{\text{case\,\,b}}
\mu&=1,3,4.
\\
\pi= 1+4+6 \rightsquigarrow 0 + \pi = 0 + 1+4+6
\rightsquigarrow \tilde{\mu} = 0 , -1,0,0 \rightsquigarrow^{\text{case\,\,b}}
\mu &= 1,2,2.\\
\pi= 2+6+7 \rightsquigarrow 0 + \pi = 0 + 2+6+7
\rightsquigarrow \tilde{\mu} = 0 , 0,2,1 \rightsquigarrow^{\text{case\,\,b}}
\mu&=2,4,3.
\\
\pi= 3+5+7 \rightsquigarrow 0 + \pi = 0 + 3+5+7
\rightsquigarrow \tilde{\mu} = 0 , 1,1,1 \rightsquigarrow^{\text{case\,\,b}}
\mu&=3,3,3.
\end{align}

Now, it turns out that the $\mu$ corresponding to case a have the following form. The maximal block corresponding to $0$ has the form $[0,-1]^+$, 
the maximal block corresponding to an even number $j\geq 2$ is
$[j,j-1]^*,j^\bullet,j^\bullet$ and for an odd $j\geq 3$ it is $j^*$. The generating function for such $\mu$ is:
\begin{align*}
\dfrac{xq^0\cdot xq^{-1}}{1-xq^0\cdot xq^{-1}}
\left(\prod\limits_{j\,\,\text{even}, j\geq 2}
\dfrac{1}{1-xq^j\cdot xq^{j-1}}
(1+xq^j + xq^j\cdot xq^j)
\right)
\left(\prod\limits_{k\,\,\text{odd}, j\geq 3}
\dfrac{1}{1-xq^k}
\right)\\
=x^2q^{-1}\left(x^2q^{-1};q^4\right)_\infty^{-1}
\left(xq^2;q \right)_\infty^{-1}
\left(x^3q^6;q^6 \right)_\infty.
\end{align*}

The $\mu$ corresponding to case b have the following form.
The initial segment of $\mu$ is either
$[1,2]^*,2^\bullet$ or $[1,2]^*,[1,3]^\bullet$. After this initial segment, we have blocks corresponding to $j\geq 3$, which are the same as before, namely, for even $j\geq 2$ we have
$[j,j-1]^*,j^\bullet,j^\bullet$ and for an odd $j\geq 3$ it is $j^*$.
The generating function for such $\mu$ is therefore the following:
\begin{align*}
\left( \dfrac{1}{1-xq\cdot xq^2}(1 + xq^2 + xq\cdot xq^3) \right)
&\left(\prod\limits_{j\,\,\text{even}, j\geq 4}
\dfrac{1}{1-xq^j\cdot xq^{j-1}}
(1+xq^j + xq^j\cdot xq^j)
\right)
\left(\prod\limits_{k\,\,\text{odd}, j\geq 3}
\dfrac{1}{1-xq^k}
\right)\\
&=\left(x^2q^{3};q^4\right)_\infty^{-1}
\left(xq^2;q \right)_\infty^{-1}
\left(x^3q^6;q^6 \right)_\infty.
\end{align*}
Combining the cases, we arrive at:
\begin{align}
x^2q^{-1}&\left(x^2q^{-1};q^4\right)_\infty^{-1}
\left(xq^2;q \right)_\infty^{-1}
\left(x^3q^6;q^6 \right)_\infty
+\left(x^2q^{3};q^4\right)_\infty^{-1}
\left(xq^2;q \right)_\infty^{-1}
\left(x^3q^6;q^6 \right)_\infty\\
&=
\left(x^2q^{-1};q^4\right)_\infty^{-1}
\left(xq^2;q \right)_\infty^{-1}
\left(x^3q^6;q^6 \right)_\infty
\\
 &=\left(\sum_{i\geq 0}\dfrac{x^iq^{2i}}{(q;q)_i}\right)
\left(\sum_{j\geq 0}\dfrac{x^{2j}q^{-j}}{(q^4;q^4)_j}\right)
\left(\sum_{k\geq 0}(-1)^k\dfrac{x^{3k}q^{3k^2+3k}}{(q^6;q^6)_k}\right)\nonumber\\
&=\sum_{i,j,k\geq 0}(-1)^k\dfrac{x^{i+2j+3k}q^{2i-j+3k^2+3k}}{(q;q)_i(q^4;q^4)_j(q^6;q^6)_k}. 
\end{align}
Now letting $x^m\mapsto x^mq^{m(m-1)}$ we get:
\begin{align}
\sum_{i,j,k\geq 0}(-1)^k\dfrac{x^{i+2j+3k}q^{(i+2j+3k)(i+2j+3k-1)+2i-j+3k^2+3k}}{(q;q)_i(q^4;q^4)_j(q^6;q^6)_k}.
\end{align}

\begin{rmk}
	This is the only identity in this paper where $x$ does not
	exactly correspond to the number of parts in the sense that for certain partitions (namely those from case a) one has to include a fictitious zero and therefore the power of $x$ has to be shifted by $1$.
\end{rmk}

\subsubsection{Identities 5 and 5a: \eqref{conj:new2} and \eqref{conj:new2a} }
The sum-side of Conjecture \eqref{conj:new2} has the following difference conditions:
\begin{enumerate}
	\item Adjacent parts do not differ by 1.
	\item Even parts do not repeat.
	\item A sub-partition of type $(2j+1) + (2j+1) + (2j+1 + t)$ is not allowed if $t\leq 3$.
	\item A sub-partition of type $(2j+1-t) + (2j+1) + (2j+1)$ is not allowed if $t\leq 2$.
	\item A sub-partition of type $(2j+1)  + (2j+3) + (2j+5)$ is not allowed.
\end{enumerate}
The difference conditions for Conjecture \eqref{conj:new2a} are the same, except for an additional initial condition:
\begin{enumerate}[resume]
	\item No 1s allowed. Alternately, assume that the partition starts with a fictitious $0$.
\end{enumerate}

Let $\pi$ be a counted in the sum-side for Identity $5$. Let $\mu$ be obtained by deleting a $2$-staircase from $\pi$.
Following patterns are forbidden in $\mu$:
\begin{enumerate}
	\item $j, j-1$ for any $j$.
	\item $j, j-2$ for even $j$.
	\item $j, j-2,j-4$ for odd $j$.
	\item $j, j-2, j-3$ for odd $j$.
	\item $j, j-2, j-2$ for odd $j$.
	\item $j, j-2, j-1$ for odd $j$.
	\item $j, j, j-2$ for odd $j$.	
	\item $j, j, j$ for odd $j$.
\end{enumerate}
Therefore, the maximal block in $\mu$ corresponding to an even $j$ is of the form $j^*$, while
the one for an odd $j$ is of the form $[j,j-2,]^*,j^\bullet,j^\bullet$.
We have that the generating function for such $\mu$ is:
\begin{align}
&\prod_{i \mathrm{\,even}}\dfrac{1}{1-xq^i}
\prod_{j \mathrm{\,odd}}\dfrac{1}{1-xq^i\cdot xq^{i-2}}
\prod_{k \mathrm{\,odd}}(1+xq^k + xq^k\cdot xq^k)\\
&=\left(xq,q\right)_\infty^{-1}\left(x^2,q^4\right)_\infty^{-1}\left(x^3q^3;q^6\right)_\infty\nonumber\\
&=\left( \sum_{i\geq 0} \dfrac{x^iq^i}{(q;q)_i}\right)\left(\sum_{j\geq 0}\dfrac{x^{2j}}{(q^4;q^4)_j}\right)\left(\sum_{k\geq 0}(-1)^k\dfrac{x^{3k}q^{3k^2}}{(q^6;q^6)_k}\right)
\nonumber\\
&= \sum_{i,j,k\geq 0}(-1)^k \dfrac{x^{i+2j+3k}q^{i+3k^2}}{(q;q)_i(q^4;q^4)_j(q^6;q^6)_k}.
\end{align}
Reinstating the $2$-staircase, we obtain:
\begin{align}
\sum_{i,j,k\geq 0}(-1)^k \dfrac{x^{i+2j+3k}
	q^{(i+2j+3k)(i+2j+3k-1)+i+3k^2}
}{(q;q)_i(q^4;q^4)_j(q^6;q^6)_k}.
\end{align}

For Identity 5a, we simply change the initial conditions to get that the generating function for $\mu$ is:
\begin{align}
&\prod_{i \mathrm{\,even}}\dfrac{1}{1-xq^i}
\prod_{j \mathrm{\,odd}, j\geq 3}\dfrac{1}{1-xq^i\cdot xq^{i-2}}
\prod_{k \mathrm{\,odd}, k\geq 3}(1+xq^k + xq^k\cdot xq^k)\\
&=\left(xq^2,q\right)_\infty^{-1}\left(x^2q^4,q^4\right)_\infty^{-1}\left(x^3q^9;q^6\right)_\infty\nonumber\\
&=\left( \sum_{i\geq 0} \dfrac{x^iq^{2i}}{(q;q)_i}\right)\left(\sum_{j\geq 0}\dfrac{x^{2j}q^{4j}}{(q^4;q^4)_j}\right)\left(\sum_{k\geq 0}(-1)^k\dfrac{x^{3k}q^{3k^2+6k}}{(q^6;q^6)_k}\right)
\nonumber\\
&= \sum_{i,j,k\geq 0}(-1)^k \dfrac{x^{i+2j+3k}q^{2i+4j+3k^2+6k}}{(q;q)_i(q^4;q^4)_j(q^6;q^6)_k}.
\end{align}
Reinstating the $2$-staircase, we obtain:
\begin{align}
\sum_{i,j,k\geq 0}(-1)^k \dfrac{x^{i+2j+3k}q^{(i+2j+3k)(i+2j+3k-1)+2i+4j+3k^2+6k}}{(q;q)_i(q^4;q^4)_j(q^6;q^6)_k}.
\end{align}

\subsubsection{Identities 6 and 6a: \eqref{conj:new3} and \eqref{conj:new3a}}
The sum-side of Conjecture \eqref{conj:new3} has the following difference conditions:
\begin{enumerate}
	\item No parts repeat.
	\item Adjacent parts do not differ by 1 if the smaller part is even. 
	\item A sub-partition of type $(2j)  + (2j+2) + (2j+4)$ is not allowed. 
	\item A sub-partition of type $(2j+1) + (2j+2) + (2j+4)$ is not allowed. 
	\item A sub-partition of type $(2j+1) + (2j+3) + (2j+4)$ is not allowed. 
\end{enumerate}
The difference conditions for Conjecture \eqref{conj:new3a} are the same, except for an additional initial condition:
\begin{enumerate}[resume]
	\item Smallest part is at least 3.
\end{enumerate}

Let $\pi$ be a counted in the sum-side for Identity $6$. Let $\mu$ be obtained by deleting a $2$-staircase from $\pi$.
Following patterns are forbidden in $\mu$:
\begin{enumerate}
	\item $j,j-2$ for any $j$.
	\item $j,j-1$ if $j$ is even.
	\item $j,j,j$ if $j$ is even.
	\item $j,j-1,j-1$ if $j$ is odd.
	\item $j,j,j-1$ if $j$ is odd.
\end{enumerate}
It is clear that the maximal block in $\mu$ corresponding to an even $j$ is $j^\bullet,j^\bullet$ while that corresponding to an odd $j$ is $[j,j-1]^*,j^*$.
Thus, the generating function for $\mu$ is:
\begin{align}
&\prod_{i\mathrm{\,even}}(1+xq^i+xq^i\cdot xq^i)
\prod_{j\mathrm{\,odd}}\dfrac{1}{(1-xq^j\cdot x q^{j-1})(1-xq^j) }\\
&=\left(xq;q\right)_\infty^{-1}\left(x^2q;q^4\right)_{\infty}^{-1}\left(x^3q^6;q^6\right)_\infty
\nonumber\\
&=\left(\sum_{i\geq 0}\dfrac{x^iq^i}{(q;q)_i}\right)
\left(\sum_{j\geq 0}\dfrac{x^{2j}q^j}{(q^4;q^4)_j}\right)\left(\sum_{k\geq 0}(-1)^k\dfrac{x^{3k}q^{3k^2+3k}}{(q^6;q^6)_k}\right)\nonumber\\
&=\sum_{i,j,k\geq 0}(-1)^k\dfrac{x^{i+2j+3k}q^{i+j+3k^2+3k}}{(q;q)_i(q^4;q^4)_j(q^6;q^6)_k}
\end{align}
Letting $x^m\mapsto x^mq^{m(m-1)}$:
\begin{align}
\sum_{i,j,k\geq 0}(-1)^k\dfrac{x^{i+2j+3k}q^{(i+2j+3k)(i+2j+3k-1)+i+j+3k^2+3k}}{(q;q)_i(q^4;q^4)_j(q^6;q^6)_k}
\end{align}

For $6a$ we incorporate the initial conditions:
\begin{align}
&\prod_{i\mathrm{\,even}, i\geq 4}(1+xq^i+xq^i\cdot xq^i)
\prod_{j\mathrm{\,odd},j\geq 3}\dfrac{1}{(1-xq^j\cdot x q^{j-1})(1-xq^j) }\\
&=\left(xq^3;q\right)_\infty^{-1}\left(x^2q^5;q^4\right)_{\infty}^{-1}\left(x^3q^{12};q^6\right)_\infty
\nonumber\\
&=\left(\sum_{i\geq 0}\dfrac{x^iq^{3i}}{(q;q)_i}\right)
\left(\sum_{j\geq 0}\dfrac{x^{2j}q^{5j}}{(q^4;q^4)_j}\right)\left(\sum_{k\geq 0}(-1)^k\dfrac{x^{3k}q^{3k^2+9k}}{(q^6;q^6)_k}\right)\nonumber\\
&=\sum_{i,j,k\geq 0}(-1)^k\dfrac{x^{i+2j+3k}q^{3i+5j+3k^2+9k}}{(q;q)_i(q^4;q^4)_j(q^6;q^6)_k}
\end{align}
Letting $x^m\mapsto x^mq^{m(m-1)}$:
\begin{align}
\sum_{i,j,k\geq 0}(-1)^k\dfrac{x^{i+2j+3k}q^{(i+2j+3k)(i+2j+3k-1)+3i+5j+3k^2+9k}}{(q;q)_i(q^4;q^4)_j(q^6;q^6)_k}.
\end{align}

\section{Analytic sum-sides for some previous conjectures from \cite{KR} and \cite{R}}

In \cite{KR} we had conjectured six new partition identities and three further conjectures were given by one of the authors in \cite{R}.
We now provide analytic sum-sides for some of these identities.
For all the identities in this section, we shall use $1$-staircases to arrive at the analytic sum-sides.
To describe these identities, we first need a few definitions.

\begin{defi}
A partition $\pi=\lambda_1+\lambda_2+\cdots+\lambda_m$ written in a weakly increasing order
is said to have ``difference at least $k$ at distance $d$'' if for all $j$, $\lambda_{j+d}-\lambda_j\geq k$.
For example, the first Rogers-Ramanujan identity recalled above enumerates partitions satisfying 
difference at least $2$ at distance $1$.	
\end{defi}

\begin{defi} Fix an integer $i$.
	A partition $\pi=\lambda_1+\lambda_2+\cdots+\lambda_m$ is said to satisfy ``Condition($i$)''
	if $\pi$ has difference at least $3$ at distance $3$
	such that if parts at distance two differ by at most $1$, then their sum (together with the intermediate part) is congruent to $i \pmod 3$.
\end{defi}

\subsection{Analytic forms for some identities from \cite{KR}}
We consider identities \cite[$I_5$]{KR} and \cite[$I_6$]{KR}.
\subsubsection{Identity $I_5$ from \cite{KR}}
This conjectural identity states that:
\begin{quote}
	The number of partitions of a non-negative integer into parts congruent to $1,$ $3,$ $4,$ $6,$ $7,$ $10,$ or $11$ (mod $12$)
	is the same as the number of partitions satisfying Condition($1$)
 with	at most one appearance of the part 1.
\end{quote}

The partition-theoretic sum-side 
has the following equivalent formulation as counting the partitions
$\pi$ which forbid the following patterns:
\begin{enumerate}
	\item $i+i+i$ for any $i$.
	\item $i+(i+1)+(i+1)$ for any $i$.
\item $i+i+(i+1)+(i+2)$ for any $i$.
\item $i+i+(i+2)+(i+2)$ for any $i$.
\item $1$ is forbidden to appear more than once.
\end{enumerate}
Consider such a partition $\pi$  and delete a $1$-staircase to obtain a jagged partition $\mu$.
$\mu$ forbids the following patterns:
\begin{enumerate}
	\item $j,j-1,j-2$ for any $j$.
	\item $j,j,j-1$ for any $j$.
\item $j,j-1,j-1,j-1$ for any $j$.	
\item $j,j-1,j,j-1$ for any $j$.	
\item $\mu$ can not begin with $1,0$.
\end{enumerate}
We get that the maximal block corresponding to any $j\geq 2$ in $\mu$ is of the form $[j,j-1,j-1]^*,[j,j-1]^\bullet, j^*$.
The maximal block corresponding to $1$ is $1^*$.
We have the following generating function for such jagged partitions $\mu$.
\begin{align}
&\dfrac{1}{1-xq}\prod_{j\geq 2}\left(\dfrac{1}{1-xq^j\cdot xq^{j-1}\cdot xq^{j-1}}(1+xq^j\cdot xq^{j-1})\dfrac{1}{1-xq^j}\right)\\
&=\left(xq;q\right)_\infty^{-1}\left(-x^2q^{3};q^2\right)\left(x^3q^4;q^3\right)_\infty^{-1}\nonumber\\
&=\left(\sum_{i\geq 0}\dfrac{x^iq^i}{(q;q)_i}\right)
\left(\sum_{j\geq 0}\dfrac{x^{2j}q^{j^2+2j}}{(q^2;q^2)_j}\right)\left(\sum_{k\geq 0}\dfrac{x^{3k}q^{4k}}{(q^3;q^3)_k}\right)\nonumber\\
&=\sum_{i,j,k\geq 0}\dfrac{x^{i+2j+3k}q^{i+j^2+2j+4k}}{(q;q)_i(q^2;q^2)_j(q^3;q^3)_k}.
\end{align}
Reinstating the $1$-staircase, i.e., $x^m\mapsto x^mq^{m(m-1)/2}$:
\begin{align}
\sum_{i,j,k\geq 0}\dfrac{x^{i+2j+3k}q^{
	\frac{(i+2j+3k)(i+2j+3k-1)}{2} + i + j^2 + 2j + 3k	
	}}{(q;q)_i(q^2;q^2)_j(q^3;q^3)_k}.
\end{align}

\subsubsection{Identity $I_6$ from \cite{KR}}
This conjectural identity states that:
\begin{quote}
The number of partitions of a non-negative integer into parts congruent to $2,$ $3,$ $5,$ $6,$ $7,$ $8,$ or $11$ (mod $12$)
is the same as the number of partitions 
satisfying Condition($2$) 
with smallest part at least 2 and at most one appearance of the part 2.
\end{quote}

The partition-theoretic sum-side has the following equivalent formulation as counting the partitions
$\pi$ which forbid the following patterns:
\begin{enumerate}
	\item $i+i+i$ for any $i$.
	\item $i+i+(i+1)$ for any $i$.
	\item $i+(i+1)+(i+2)+(i+2)$ for any $i$.
	\item $i+i+(i+2)+(i+2)$ for any $i$.
	\item $1$ is forbidden to appear.
	\item $2$ can appear at most once.
\end{enumerate}
Consider such a partition $\pi$ and delete a $1$-staircase to obtain a jagged partition $\mu$.
$\mu$ forbids the following patterns:
\begin{enumerate}
	\item $j,j-1,j-2$ for any $j$.
	\item $j,j-1,j-1$ for any $j$.
	\item $j,j,j,j-1$ for any $j$.	
	\item $j,j-1,j,j-1$ for any $j$.	
	\item $\mu$ can not begin with $1$.
	\item $\mu$ can not begin with $2,1$.
\end{enumerate}
We get that the maximal block corresponding to any $j\geq 3$ in $\mu$ is of the form $[j,j-1]^\bullet,[j,j,j-1]^*, j^*$.
The maximal block corresponding to $2$ is $[2,2,1]^*,2^*$.
We have the following generating function for such jagged partitions $\mu$.
\begin{align}
&\dfrac{1}{1-xq^2\cdot xq^2\cdot xq}\dfrac{1}{1-xq^2}
\prod_{j\geq 3}\left((1+xq^j\cdot xq^{j-1})\dfrac{1}{1-xq^j\cdot xq^{j}\cdot xq^{j-1}}\dfrac{1}{1-xq^j}\right)\\
&=\left(xq^2;q\right)_\infty^{-1}\left(-x^2q^{5};q^2\right)\left(x^3q^5;q^3\right)_\infty^{-1}\label{KRI6:removestair}\\
&=\left(\sum_{i\geq 0}\dfrac{x^iq^{2i}}{(q;q)_i}\right)
\left(\sum_{j\geq 0}\dfrac{x^{2j}q^{j^2+4j}}{(q^2;q^2)_j}\right)\left(\sum_{k\geq 0}\dfrac{x^{3k}q^{5k}}{(q^3;q^3)_k}\right)\nonumber\\
&=\sum_{i,j,k\geq 0}\dfrac{x^{i+2j+3k}q^{2i+j^2+4j+5k}}{(q;q)_i(q^2;q^2)_j(q^3;q^3)_k}.
\end{align}
Reinstating the $1$-staircase, i.e., $x^m\mapsto x^mq^{m(m-1)/2}$:
\begin{align}
\sum_{i,j,k\geq 0}\dfrac{x^{i+2j+3k}q^{
	\frac{(i+2j+3k)(i+2j+3k-1)}{2}	+ 2i + j^2+4j + 5k
		}
	}{(q;q)_i(q^2;q^2)_j(q^3;q^3)_k}.
\end{align}

With a look ahead towards finding further companions to this identity, we now present a different analytic sum side
to \cite[$I_6$]{KR}.
Let us rewrite \eqref{KRI6:removestair} as follows:
\begin{align}
\left(xq^2;q\right)_\infty^{-1}&\left(-x^2q^{5};q^2\right)\left(x^3q^5;q^3\right)_\infty^{-1}
=\left(xq^2;q\right)_\infty^{-1}\dfrac{\left(x^4q^{10};q^4\right)_\infty}{\left(x^2q^5;q^2\right)_\infty}\left(x^3q^5;q^3\right)_\infty^{-1}\\
&=\left(\sum\limits_{i\geq 0}\dfrac{x^iq^{2i}}{(q;q)_i} \right)
\left(\sum\limits_{j\geq 0}\dfrac{x^{2j}q^{5j}}{\left(q^2;q^2\right)_j} \right)
\left(\sum\limits_{k\geq 0}\dfrac{x^{3k}q^{5k}}{\left(q^3;q^3\right)_k} \right)
\left(\sum\limits_{l\geq 0}(-1)^l\dfrac{x^{4l}q^{2l^2+8l}}{\left(q^4;q^4\right)_l} \right)
\nonumber\\
&=\sum\limits_{i,j,k,l\geq 0}(-1)^l\dfrac{x^{i+2j+3k+4l}q^{2i+5j+5k+2l^2+8l}}{(q;q)_i\left(q^2;q^2\right)_j\left(q^3;q^3\right)_k\left(q^4;q^4\right)_l}.
\end{align}
Putting back the $1$-staircase, i.e., $x^m\mapsto x^mq^{m(m-1)/2}$ we arrive at an alternate sum-side:
\begin{align}
\sum\limits_{i,j,k,l\geq 0}(-1)^l\dfrac{x^{i+2j+3k+4l}q^{\frac{(i+2j+3k+4l)(i+2j+3k+4l-1)}{2}+ 2i+5j+5k+2l^2+8l}}{(q;q)_i\left(q^2;q^2\right)_j\left(q^3;q^3\right)_k\left(q^4;q^4\right)_l}.
\label{eqn:KRI6alternate}
\end{align}
Below, we shall vary the linear term in the exponent of $q$ to deduce further companions to this identity.

\subsection{Analytic forms for some identities from \cite{R}}

Motivated by the desire to find certain complementary identities to the conjectures in \cite{KR}, one of the authors
 provided three further conjectures in \cite{R}. Now, we provide analytic sum-sides to \cite[$I_{5a}$]{R}
and \cite[$I_{6a}$]{R}.
\subsubsection{Identity $I_{5a}$ from \cite{R}}
This conjectural identity states that:
\begin{quote}
	The number of partitions of a non-negative integer into parts congruent to $1, 2,
	5, 6, 8, 9,$ or $11$ (mod $12$) is the same as the number of partitions 
	satisfying Condition($2$) such that $1+2+2$ is not allowed to appear in the partition.
\end{quote}

The conditions on the partition-theoretic sum-side of \cite[$I_{5a}$]{R} 
are the same as the one for \cite[$I_6$]{KR}, except for the initial condition:
\cite[$I_{5a}$]{R} forbids the appearance of $1+2+2$ in $\pi$.
This translates to forbidding $1,1,0$ as an initial segment in $\mu$.
In effect, the block corresponding to $1$ in $\mu$ has to be $[1,0,[1,1,0]^*]^\bullet,1^*$.
 Therefore, modifying what we have above appropriately, 
we have the following generating function for such jagged partitions $\mu$.
\begin{align}
&\left(1+\dfrac{x^2q}{1-x^3q^2}\right)\dfrac{1}{1-xq}
\prod_{j\geq 2}\left((1+xq^j\cdot xq^{j-1})\dfrac{1}{1-xq^j\cdot xq^{j}\cdot xq^{j-1}}\dfrac{1}{1-xq^j}\right)\\
&=(1+x^2q-x^3q^2)
\left(xq;q\right)_\infty^{-1}\left(-x^2q^3;q^2\right)\left(x^3q^2;q^3\right)_\infty^{-1}\nonumber\\
&=(1+x^2q-x^3q^2)\left(\sum_{i\geq 0}\dfrac{x^iq^{i}}{(q;q)_i}\right)
\left(\sum_{j\geq 0}\dfrac{x^{2j}q^{j^2+2j}}{(q^2;q^2)_j}\right)\left(\sum_{k\geq 0}\dfrac{x^{3k}q^{2k}}{(q^3;q^3)_k}\right)\nonumber\\
&=(1+x^2q-x^3q^2)\left(\sum_{i,j,k\geq 0}\dfrac{x^{i+2j+3k}q^{i+j^2+2j+5k}}{(q;q)_i(q^2;q^2)_j(q^3;q^3)_k}\right)
\nonumber\\
&=
\sum_{i,j,k\geq 0}
\dfrac{x^{i+2j+3k}q^{i+j^2+2j+5k}+x^{i+2j+3k+2}q^{i+j^2+2j+5k+1}-x^{i+2j+3k+3}q^{i+j^2+2j+2k+2}}{(q;q)_i(q^2;q^2)_j(q^3;q^3)_k}
\end{align}
Reinstating $1$-staircase, we have:
\begin{align}
\sum_{i,j,k\geq 0}
&\left\lbrace
\dfrac{
	x^{i+2j+3k}q^{\frac{(i+2j+3k)(i+2j+3k-1)}{2}+i+j^2+2j+2k}
}{(q;q)_i(q^2;q^2)_j(q^3;q^3)_k}
\right. \nonumber\\
&\qquad+\dfrac{
	x^{i+2j+3k+2}q^{\frac{(i+2j+3k+2)(i+2j+3k+1)}{2}+i+j^2+2j+2k+1}
}{(q;q)_i(q^2;q^2)_j(q^3;q^3)_k}
\nonumber\\
&\qquad-\left.
\dfrac{
		x^{i+2j+3k+3}q^{\frac{(i+2j+3k+3)(i+2j+3k+2)}{2}+i+j^2+2j+2k+2}
}{(q;q)_i(q^2;q^2)_j(q^3;q^3)_k}
\right\rbrace.
\end{align}

\subsubsection{Identity $I_{6a}$ from \cite{R}} 
This conjectural identity states that:
\begin{quote}
	The number of partitions of a non-negative integer into parts congruent to $1, 4,
	5, 6, 7, 9$, or $10$ (mod $12$) is the same as the number of partitions 
	satisfying Condition($1$) such that
	2 is not allowed to appear in the partition.
\end{quote}

The sum-side conditions on this identity are same as the one for \cite[$I_5$]{KR}, except for the initial conditions.
For \cite[$I_{6a}$]{R}, $2$ is forbidden to appear as a part.
Therefore, proceeding just like \cite[$I_5$]{KR} above, the maximal block in $\mu$ corresponding to $j\geq 3$, is of the form $[j,j-1,j-1]^*,[j,j-1]^\bullet,j^*$. However, if $1$ appears in $\mu$, the initial block in $\mu$ has to be either:
$1,0, 1^*,[2,1,1]^*,[2,1]^\bullet,2^*$ or
$1,[2,1,1]^*,[2,1]^\bullet, 2^*$.

The generating function for $\mu$ is:
\begin{align}
&
\left(  \left(xq + xq\cdot x\cdot\dfrac{1}{1-xq}\right)\dfrac{1}{1-xq^2\cdot xq\cdot xq}
(1+xq^2\cdot xq) \dfrac{1}{1-xq^2}
 +1\right)\cdot \nonumber\\
&\quad\prod_{j\geq 3}\left(\dfrac{1}{1-xq^j\cdot xq^{j-1}\cdot xq^{j-1}}(1+xq^j\cdot xq^{j-1})\dfrac{1}{1-xq^j}\right)\\
&=\left({\frac {xq \left(1+ {x}^{2}{q}^{3} \right) }{ \left(1- {x}^{3}{q}^{4}\right) \left(1- x{q}^{2} \right) }}
+{\frac{{x}^{2}q \left(1+ {x}^{2}{q}^{3} \right) }{ \left( 1-xq\right)  \left(1- {x}^{3}{q}^{4} \right)  \left(1- x{q}^{2} \right) }}
   +1\right)
\left(xq^3;q\right)_\infty^{-1}\left(-x^2q^{5};q^2\right)\left(x^3q^7;q^3\right)_\infty^{-1}\nonumber\\
&=
xq\left(xq^2;q\right)_\infty^{-1}\left(-x^2q^{3};q^2\right)\left(x^3q^4;q^3\right)_\infty^{-1}
+x^2q\left(xq;q\right)_\infty^{-1}\left(-x^2q^{3};q^2\right)\left(x^3q^4;q^3\right)_\infty^{-1}\nonumber\\
&\quad +\left(xq^3;q\right)_\infty^{-1}\left(-x^2q^{5};q^2\right)\left(x^3q^7;q^3\right)_\infty^{-1}
\nonumber\\
&=\sum_{i,j,k\geq 0}
\dfrac{x^{i+2j+3k+1}q^{2i+j^2+2j+4k+1}
+ x^{i+2j+3k+2}q^{i+j^2+2j+4k+1}
+ x^{i+2j+3k}q^{3i+j^2+4j+7k}}
{(q;q)_i(q^2;q^2)_j(q^3;q^3)_k}.
\end{align}
Reinstating the $1$-staircase, i.e., $x^m\mapsto x^mq^{m(m-1)/2}$:
\begin{align}
\sum_{i,j,k\geq 0}
&\left\lbrace
\dfrac{x^{i+2j+3k+1}q^{\frac{(i+2j+3k+1)(i+2j+3k)}{2}
		+2i+j^2+2j+4k+1}
}
{(q;q)_i(q^2;q^2)_j(q^3;q^3)_k}
\right.\nonumber\\
&\quad +
\dfrac{
		 x^{i+2j+3k+2}q^{\frac{(i+2j+3k+2)(i+2j+3k+1)}{2}+i+j^2+2j+4k+1}
}
{(q;q)_i(q^2;q^2)_j(q^3;q^3)_k}\nonumber\\
&\left.\quad +
\dfrac{
 x^{i+2j+3k}q^{\frac{(i+2j+3k)(i+2j+3k-1)}{2}+3i+j^2+4j+7k}
}
{(q;q)_i(q^2;q^2)_j(q^3;q^3)_k}\right\rbrace.
\end{align}

\section{Further companions of some previous conjectures from \cite{KR} and \cite{R}}
As promised above, we now vary the linear term in the exponent of $q$ in the expression \eqref{eqn:KRI6alternate}
to conjecture further companions to \cite[$I_6$]{KR}.

\subsection{Analytic forms}
We have the following conjectured identities.
\begin{align}
\sum\limits_{i,j,k,l\geq 0}(-1)^l\dfrac{q^{\frac{(i+2j+3k+4l)(i+2j+3k+4l-1)}{2}+2l^2+ i+3j+6k+6l}}{(q;q)_i\left(q^2;q^2\right)_j\left(q^3;q^3\right)_k\left(q^4;q^4\right)_l}
&=\dfrac{1}{\left(q,q^{3},q^{4},q^{6},q^{8},q^{9},q^{11} ;q^{12}\right)_\infty},
\label{eqn:KRI7}\\
\sum\limits_{i,j,k,l\geq 0}(-1)^l\dfrac{q^{\frac{(i+2j+3k+4l)(i+2j+3k+4l-1)}{2}+2l^2+ 3i+5j+6k+10l}}{(q;q)_i\left(q^2;q^2\right)_j\left(q^3;q^3\right)_k\left(q^4;q^4\right)_l}
&=\dfrac{1}{\left(q^3,q^{4},q^{5},q^{6},q^{7},q^{8},q^{9} ;q^{12}\right)_\infty},
\label{eqn:KRI7a}\\
\sum\limits_{i,j,k,l\geq 0}(-1)^l\dfrac{q^{\frac{(i+2j+3k+4l)(i+2j+3k+4l-1)}{2}+2l^2+ 2i+3j+5k+6l}}{(q;q)_i\left(q^2;q^2\right)_j\left(q^3;q^3\right)_k\left(q^4;q^4\right)_l}
&=\dfrac{1}{\left(q^2,q^{3},q^{4},q^{5},q^{8},q^{9},q^{11} ;q^{12}\right)_\infty}.
\label{eqn:KRI8}
\end{align}
Note that the product-sides of \eqref{eqn:KRI7} and \eqref{eqn:KRI7a} are symmetric.
Search for an identity in which the allowable congruences in the product-side are negatives of those
appearing in \eqref{eqn:KRI8} results in the following identity. 
\begin{align}
\sum_{i,j,k,l\geq 0}
&\left\lbrace 
(-1)^l\dfrac{
q^{\frac{(i+2j+3k+4l)(i+2j+3k+4l-1)}{2} + 2l^2
	+ 2i + 3j + 4k + 6l}
}{(q;q)_i(q^2;q^2)_j(q^3;q^3)_k(q^4;q^4)_l}
+ 
(-1)^l\dfrac{
	q^{\frac{(i+2j+3k+4l+1)(i+2j+3k+4l)}{2} + 2l^2
	+ 2i + 3j + 4k + 6l+1}
}{(q;q)_i(q^2;q^2)_j(q^3;q^3)_k(q^4;q^4)_l}\right.\nonumber\\
&\left.+
(-1)^l\dfrac{
	q^{\frac{(i+2j+3k+4l+2)(i+2j+3k+4l+1)}{2} + 2l^2
		+ 2i + 3j + 4k + 6l+2}
}{(q;q)_i(q^2;q^2)_j(q^3;q^3)_k(q^4;q^4)_l}
+
(-1)^l\dfrac{
	q^{\frac{(i+2j+3k+4l+3)(i+2j+3k+4l+2)}{2} + 2l^2
		+ 2i + 3j + 4k + 6l+4}
}{(q;q)_i(q^2;q^2)_j(q^3;q^3)_k(q^4;q^4)_l}
\right\rbrace\nonumber\\
&\qquad=\dfrac{1}{\left(q,q^3,q^4,q^7,q^8,q^9,q^{10} ;q^{12}\right)_\infty}.
\label{eqn:KRI8a}
\end{align}

\subsection{Partition-theoretic sum-sides} We now present a partition-theoretic interpretation for each of the sum-sides in the identities above.
\subsubsection{Identities $7$ and $7a$: \eqref{eqn:KRI7} and \eqref{eqn:KRI7a}}
The sum-side of these identities count partitions $\pi$ forbidding the following patterns:
\begin{enumerate}
	\item $i + (i+1) + (i+1)$. 
	\item $i + i + (i+1)$.
	\item $i + (i+2) + (i+2) + (i+2)$.
	\item $i + i + i + (i+2)$.
    \item $i + i + i + i$.
\end{enumerate}
For Identity $7$, i.e., \eqref{eqn:KRI7} the initial conditions are given by a fictitious zero:
\begin{enumerate}
    \item[(6)]  $1+1$ and $2+2+2$ are forbidden to appear.
\end{enumerate}
For Identity $7a$, i.e., \eqref{eqn:KRI7a} the initial conditions are given by:
\begin{enumerate}
	\item[(6)]  Smallest part is at least $3$.
\end{enumerate}

Let us first work with Identity $7$, i.e., \eqref{eqn:KRI7}.
Remove a $1$-staircase from $\pi$ to obtain a jagged partition $\mu$.
It is clear that $\mu$ forbids the following patterns:
\begin{enumerate}
	\item $j,j,j-1$.
	\item $j,j-1,j-1$.
	\item $j,j+1,j,j-1$.
	\item $j,j-1,j-2,j-1$.
	\item $j,j-1,j-2,j-3$.
	\item $\mu$ can not start with $1,0$ and $2,1,0$.
\end{enumerate}
The maximal block in $\mu$ corresponding to a part $j$
is of the shape $[j,j-1,j-2]^*,[j,j-1]^*,j^*$, with the following exception:
we may not have a string of $j,j+1,j,j-1$. 
Moreover, due to the initial conditions on $\pi$, the blocks corresponding to $1$ and $2$ in $\mu$ must be of the form $1^*,[2,1]^*,2^*$.
This translates to the following generating function for $\mu$:
\begin{align}
&\dfrac{1}{1-xq}\cdot \dfrac{1}{1-xq^2\cdot xq}\cdot \dfrac{1}{1-xq^2}\cdot\nonumber \\
&\quad\quad\quad\cdot
\prod\limits_{j\geq 3}\left(\dfrac{1}{1-xq^j\cdot xq^{j-1} \cdot xq^{j-2}}\cdot 
\dfrac{1}{1-xq^j\cdot xq^{j-1}}\cdot \dfrac{1}{1-xq^j} \right)
\cdot
\prod\limits_{j\geq 2}\left(1-xq^{j}\cdot xq^{j+1}\cdot xq^{j}\cdot xq^{j-1}  
 \right)\\
&=\left(xq;q\right)_\infty^{-1}\left(x^2q^3;q^2\right)_\infty^{-1}\left(x^3q^6;q^3\right)_\infty^{-1}
\left(x^4q^8;q^4\right)_\infty\nonumber\\
&=\left(\sum\limits_{i\geq 0}\dfrac{x^iq^{i}}{(q;q)_i} \right)
\left(\sum\limits_{j\geq 0}\dfrac{x^{2j}q^{3j}}{\left(q^2;q^2\right)_j} \right)
\left(\sum\limits_{k\geq 0}\dfrac{x^{3k}q^{6k}}{\left(q^3;q^3\right)_k} \right)
\left(\sum\limits_{l\geq 0}(-1)^l\dfrac{x^{4l}q^{2l^2+6l}}{\left(q^4;q^4\right)_l} \right)
\nonumber\\
&=\sum\limits_{i,j,k,l\geq 0}(-1)^l\dfrac{x^{i+2j+3k+4l}q^{2l^2+i+3j+6k+6l}}{(q;q)_i\left(q^2;q^2\right)_j\left(q^3;q^3\right)_k\left(q^4;q^4\right)_l}.
\end{align}
Putting back the $1$-staircase, i.e., $x^m\mapsto x^mq^{m(m-1)/2}$ we get:
\begin{align}
\sum\limits_{i,j,k,l\geq 0}(-1)^l\dfrac{x^{i+2j+3k+4l}q^{\frac{(i+2j+3k+4l)(i+2j+3k+4l-1)}{2}+2l^2+ i+3j+6k+6l}}{(q;q)_i\left(q^2;q^2\right)_j\left(q^3;q^3\right)_k\left(q^4;q^4\right)_l}.
\end{align}

For Identity $7a$, i.e., \eqref{eqn:KRI7a} we merely change the initial conditions
to get the following generating function for $\mu$:
\begin{align}
&
\prod\limits_{j\geq 3}\left(\dfrac{1}{1-xq^j\cdot xq^{j-1} \cdot xq^{j-2}}\cdot 
\dfrac{1}{1-xq^j\cdot xq^{j-1}}\cdot \dfrac{1}{1-xq^j} \right)
\cdot\prod\limits_{j\geq 3}\left(1-xq^{j}\cdot xq^{j+1}\cdot xq^{j}\cdot xq^{j-1}  
\right)\\
&=\left(xq^3;q\right)_\infty^{-1}\left(x^2q^5;q^2\right)_\infty^{-1}\left(x^3q^6;q^3\right)_\infty^{-1}
\left(x^4q^{12};q^4\right)_\infty\nonumber\\
&=\left(\sum\limits_{i\geq 0}\dfrac{x^iq^{3i}}{(q;q)_i} \right)
\left(\sum\limits_{j\geq 0}\dfrac{x^{2j}q^{5j}}{\left(q^2;q^2\right)_j} \right)
\left(\sum\limits_{k\geq 0}\dfrac{x^{3k}q^{6k}}{\left(q^3;q^3\right)_k} \right)
\left(\sum\limits_{l\geq 0}(-1)^l\dfrac{x^{4l}q^{2l^2+10l}}{\left(q^4;q^4\right)_l} \right)
\nonumber\\
&=\sum\limits_{i,j,k,l\geq 0}(-1)^l\dfrac{x^{i+2j+3k+4l}q^{2l^2+3i+5j+6k+10l}}{(q;q)_i\left(q^2;q^2\right)_j\left(q^3;q^3\right)_k\left(q^4;q^4\right)_l}.
\end{align}
Putting back the $1$-staircase, i.e., $x^m\mapsto x^mq^{m(m-1)/2}$ we get:
\begin{align}
\sum\limits_{i,j,k,l\geq 0}(-1)^l\dfrac{x^{i+2j+3k+4l}q^{\frac{(i+2j+3k+4l)(i+2j+3k+4l-1)}{2}+2l^2 +3i+5j+6k+10l}}{(q;q)_i\left(q^2;q^2\right)_j\left(q^3;q^3\right)_k\left(q^4;q^4\right)_l}.
\end{align}

\subsubsection{Identity $8$: \eqref{eqn:KRI8}}
The sum-side of this identity counts partitions $\pi$ forbidding the following patterns:
\begin{enumerate}
	\item $i+i+i$.
	\item $i+i+(i+1)$.
	\item $i+(i+1)+(i+2)+(i+2)$.
	\item $i+(i+1)+(i+2)+(i+3)$.
	\item $i+(i+1)+(i+1)+(i+3)+(i+3)$.
	\item Initial conditions are given by two fictitious zeros, i.e., $1$ is forbidden to appear as a part.
\end{enumerate}
Removing a staircase, we see that $\mu$ must forbid the following patters:
\begin{enumerate}
	\item $j,j-1,j-2$.
	\item $j,j-1,j-1$.
	\item $j,j,j,j-1$.
	\item $j,j,j,j$.
	\item $j,j,j-1,j,j-1$.
	\item $\mu$ does not start with a $1$.
\end{enumerate}
The maximal block corresponding to a part $j$ in $\mu$ is therefore of the shape
$[j,j-1]^*,[j,j,j-1]^*,j^\bullet, j^\bullet, j^\bullet$
and we get the following generating function for $\mu$:
\begin{align}
&
\prod\limits_{j\geq 2}\left(
\dfrac{1}{1-xq^j\cdot xq^{j-1}}\cdot 
\dfrac{1}{1-xq^j\cdot xq^{j-1} \cdot xq^{j-1}}
\cdot\left(1 + xq^{j} + xq^j\cdot xq^j+ xq^j\cdot xq^j\cdot xq^j\right) \right)
\\
&=\left(xq^2;q\right)_\infty^{-1}\left(x^2q^3;q^2\right)_\infty^{-1}\left(x^3q^5;q^3\right)_\infty^{-1}
\left(x^4q^{8};q^4\right)_\infty\nonumber\\
&=\left(\sum\limits_{i\geq 0}\dfrac{x^iq^{2i}}{(q;q)_i} \right)
\left(\sum\limits_{j\geq 0}\dfrac{x^{2j}q^{3j}}{\left(q^2;q^2\right)_j} \right)
\left(\sum\limits_{k\geq 0}\dfrac{x^{3k}q^{5k}}{\left(q^3;q^3\right)_k} \right)
\left(\sum\limits_{l\geq 0}(-1)^l\dfrac{x^{4l}q^{2l^2+6l}}{\left(q^4;q^4\right)_l} \right)
\nonumber\\
&=\sum\limits_{i,j,k,l\geq 0}(-1)^l\dfrac{x^{i+2j+3k+4l}q^{2l^2+2i+3j+5k+6l}}{(q;q)_i\left(q^2;q^2\right)_j\left(q^3;q^3\right)_k\left(q^4;q^4\right)_l}.
\end{align}
Putting back the $1$-staircase, i.e., $x^m\mapsto x^mq^{m(m-1)/2}$ we have:
\begin{align}
\sum\limits_{i,j,k,l\geq 0}(-1)^l\dfrac{x^{i+2j+3k+4l}q^{\frac{(i+2j+3k+4l)(i+2j+3k+4l-1)}{2}+2l^2 +2i+3j+5k+6l}}{(q;q)_i\left(q^2;q^2\right)_j\left(q^3;q^3\right)_k\left(q^4;q^4\right)_l}.
\end{align}

\subsubsection{Identity $8a$: \eqref{eqn:KRI8a}} 
The sum-side of this identity counts partitions forbidding the following patterns.
\begin{enumerate}
	\item $i+i+i$.
	\item $i+(i+1) + (i+1)$.
	\item $i+i + (i+1)+(i+2)$.
	\item $i+(i+1)+(i+2)+(i+3)$.
	\item $i+i+(i+2)+(i+2)+(i+3)$.
	\item $1+1$.
	\item $1+2+3$.
	\item $2+2+3$.	
\end{enumerate}
Removing a $1$-staircase, we see that $\mu$ forbids:
\begin{enumerate}
	\item $j, j-1,j-2$.
	\item $j,j,j-1$.
	\item $j,j-1,j-1,j-1$.
	\item $j,j,j,j$.
	\item $j,j-1,j,j-1,j-1$.
	\item $\mu$ does not start with $1,0$ or $1,1,1$ or $2,1,1$.
\end{enumerate} 
It is now clear that for the maximal block in $\mu$ corresponding to a part $j\geq 3$ is $[j,j-1,j-1]^*,[j,j-1]^*,j^\bullet,j^\bullet,j^\bullet$.
If $\mu$ does not start with $1$, then the block corresponding to $2$ is $[2,1]^*,2^\bullet,2^\bullet,2^\bullet$.
However, if $\mu$ does start with $1$, then the blocks corresponding to $1$ and $2$ match:
$1,1^\bullet,[2,1,1]^*,[2,1]^*,2^\bullet,2^\bullet,2^\bullet$.
Combining, we get that the generating function of such $\mu$ is:
\begin{align}
&(xq+x^2q^2)\prod_{j\geq 2}
\left(\dfrac{1}{1-xq^j\cdot xq^{j-1}\cdot xq^{j-1}}\cdot
\dfrac{1}{1-xq^j\cdot xq^{j-1}}\cdot
(1+xq^j + xq^{j}\cdot xq^j +xq^{j}\cdot xq^{j}\cdot xq^j)
\right)\nonumber\\
&+
\dfrac{1+xq^2+x^2q^4+
	x^3q^6}{1-xq^2\cdot xq}
\prod_{j\geq 3}\left(\dfrac{1}{1-xq^j\cdot xq^{j-1}\cdot xq^{j-1}}\cdot
\dfrac{1}{1-xq^j\cdot xq^{j-1}}\cdot
(1+xq^j + xq^{j}\cdot xq^j +xq^{j}\cdot xq^{j}\cdot xq^j)
\right)\\
& =
(xq+x^2q^2)\left(xq^2;q\right)_\infty^{-1}
\left(x^2q^3;q^2\right)_\infty^{-1}
\left(x^3q^4;q^3\right)_\infty^{-1}
\left(x^4q^8;q^4\right)_\infty
 +
\left(xq^2;q\right)_\infty^{-1}
\left(x^2q^3;q^2\right)_\infty^{-1}
\left(x^3q^7;q^3\right)_\infty^{-1}
\left(x^4q^8;q^4\right)_\infty\nonumber\\
&=
\left(1+xq+x^2q^2-x^3q^4\right)
\left(xq^2;q\right)_\infty^{-1}
\left(x^2q^3;q^2\right)_\infty^{-1}
\left(x^3q^4;q^3\right)_\infty^{-1}
\left(x^4q^8;q^4\right)_\infty\nonumber\\
&=\left(1+xq+x^2q^2-x^3q^4\right)
\left(\sum_{i\geq 0}\dfrac{x^iq^{2i}}{\left(q;q\right)_i} \right)
\left(\sum_{j\geq 0}\dfrac{x^{2j}q^{3j}}{\left(q^2;q^2\right)_{j}} \right)
\left(\sum_{k\geq 0}\dfrac{x^{3k}q^{4k}}{\left(q^3;q^3\right)_{k}} \right)
\left(\sum_{l\geq 0}(-1)^l\dfrac{x^{4l}q^{2l^2+6l}}{\left(q^4;q^4\right)_{l}} \right)
\nonumber\\
&=\sum_{i,j,k,l\geq 0}
\left\lbrace 
(-1)^l\dfrac{
	x^{i+2j+3k+4l}q^{2l^2
		+ 2i + 3j + 4k + 6l}
}{(q;q)_i(q^2;q^2)_j(q^3;q^3)_k(q^4;q^4)_l}
+ 
(-1)^l\dfrac{
	x^{i+2j+3k+4l+1}q^{2l^2
		+ 2i + 3j + 4k + 6l+1}
}{(q;q)_i(q^2;q^2)_j(q^3;q^3)_k(q^4;q^4)_l}\right.\nonumber\\
&\left. +
(-1)^l\dfrac{
	x^{i+2j+3k+4l+2}q^{2l^2
		+ 2i + 3j + 4k + 6l+2}
}{(q;q)_i(q^2;q^2)_j(q^3;q^3)_k(q^4;q^4)_l}
+
(-1)^l\dfrac{
	x^{i+2j+3k+4l+3}q^{2l^2
		+ 2i + 3j + 4k + 6l+4}
}{(q;q)_i(q^2;q^2)_j(q^3;q^3)_k(q^4;q^4)_l}
\right\rbrace.
\end{align}
Reinstating the $1$-staircase, we get:
\begin{align}
\sum_{i,j,k,l\geq 0}
&\left\lbrace 
(-1)^l\dfrac{
	x^{i+2j+3k+4l}q^{\frac{(i+2j+3k+4l)(i+2j+3k+4l-1)}{2}+2l^2
		+ 2i + 3j + 4k + 6l}
}{(q;q)_i(q^2;q^2)_j(q^3;q^3)_k(q^4;q^4)_l}\right.
\nonumber\\
&+ 
(-1)^l\dfrac{
	x^{i+2j+3k+4l+1}q^{\frac{(i+2j+3k+4l+1)(i+2j+3k+4l)}{2}+2l^2
		+ 2i + 3j + 4k + 6l+1}
}{(q;q)_i(q^2;q^2)_j(q^3;q^3)_k(q^4;q^4)_l}\nonumber\\
& +
(-1)^l\dfrac{
	x^{i+2j+3k+4l+2}q^{\frac{(i+2j+3k+4l+2)(i+2j+3k+4l+1)}{2}+2l^2
		+ 2i + 3j + 4k + 6l+2}
}{(q;q)_i(q^2;q^2)_j(q^3;q^3)_k(q^4;q^4)_l}\nonumber\\
&\left. +
(-1)^l\dfrac{
	x^{i+2j+3k+4l+3}q^{\frac{(i+2j+3k+4l+3)(i+2j+3k+4l+2)}{2}+2l^2
		+ 2i + 3j + 4k + 6l+4}
}{(q;q)_i(q^2;q^2)_j(q^3;q^3)_k(q^4;q^4)_l}
\right\rbrace.
\end{align}

\section{Analytic sum-sides for Capparelli's identities}

Recall Capparelli's identities \cite{C1}
which arose from level $3$ standard modules for $A_2^{(2)}$:
\begin{thm*} For any positive integer $n$ we have that:
	\begin{enumerate}
		\item Number of partitions of $n$ into parts different from $1$ 
		such that the difference of two consecutive parts is at least $2$, and is exactly $2$ or $3$ only if their sum is a multiple of $3$ is the same as the number of partitions in which every part is congruent to $\pm 2$ or  $\pm 3 \,\,(\mathrm{mod}\,12)$.
		\item Number of partitions of $n$ into parts different from $2$ 
		such that the difference of two consecutive parts is at least $2$, and is exactly $2$ or $3$ only if their sum is a multiple of $3$ is the same as the number of partitions into distinct parts congruent to $1$, $3$, $5$, or $0$ $(\text{mod}\,\, 6)$.
	\end{enumerate}
\end{thm*}

The difference conditions common to both the identities forbid consecutive differences equaling $0$ or $1$
and additionally following patterns are forbidden:
\begin{enumerate}
	\item $3j + (3j+2)$.
	\item $(3j+1) + (3j+3)$.
	\item $(3j+1) + (3j+4)$.
	\item $(3j+2) + (3j+5)$.
\end{enumerate}
For the first identity, the initial conditions stipulate that
\begin{enumerate} 
	\item[(7)] No $1$s are allowed.
\end{enumerate}
For the second identity, the initial conditions stipulate that
\begin{enumerate}
	\item[(7)] No $2$s are allowed.
\end{enumerate}

Let us work with the first identity.
What follows is essentially a ``dilated'' version of the argument in \cite{DL}.

Deleting a $3$-staircase, $\mu$ forbids the following patterns:
\begin{enumerate}
	\item $j$, $j-3$.
	\item $j$, $j-2$.
	\item $3j,3j-1$.
	\item $3j+1, 3j$.
	\item $3j+1, 3j+1$.
	\item $3j+2, 3j+2$.
\end{enumerate}
Therefore, one can conclude that the maximal block in $\mu$ corresponding to $3j$ is $(3j)^*$;
the one for $3j+1$ ($j\geq 1$) is $(3j+1)^\bullet$; for $3j+2$ it is $[(3j+2),(3j+1)]^*,(3j+2)^\bullet$.
It is now easy to deduce that the generating function for $\mu$ is:
\begin{align}
&\left(\prod\limits_{j\geq 1} \dfrac{1}{1-xq^{3j}}\right)
\left(\prod\limits_{j\geq 1}1+xq^{3j+1}\right)
\left(\prod\limits_{j\geq 0} \dfrac{1}{1-xq^{3j+2}\cdot xq^{3j+1}}\cdot (1+xq^{3j+2})\right)
\nonumber\\
&=	\left(xq^3;q^3\right)_\infty^{-1}
\left(-xq^2;q^3\right)_\infty\left(-xq^4;q^3\right)_\infty\left(x^2q^3;q^6\right)_\infty^{-1}\label{eqn:Cap1eq1}\\
&=\sum\limits_{i,j,k,l\geq 0}\dfrac{x^{i+j+k+2l}
	q^{3i + 2j + \frac{3j(j-1)}{2} + 4k + \frac{3k(k-1)}{2}  + 3l}
}{\left(q^3;q^3\right)_i\left(q^3;q^3\right)_j\left(q^3;q^3\right)_k\left(q^6;q^6\right)_l}.
\end{align}
Now we put back a $3$-staircase, i.e., $x^m\mapsto x^mq^{3m(m-1)/2}$:
\begin{align}
\sum\limits_{i,j,k,l\geq 0}\dfrac{x^{i+j+k+2l}
	q^{\frac{3(i+j+k+2l)(i+j+k+2l-1)}{2}+3i + 2j + \frac{3j(j-1)}{2} + 4k + \frac{3k(k-1)}{2}  + 3l}
}{\left(q^3;q^3\right)_i\left(q^3;q^3\right)_j\left(q^3;q^3\right)_k\left(q^6;q^6\right)_l},
\end{align}
which is exactly \cite[Equation (2.6)]{DL} with
$r\mapsto k, s\mapsto j, t\mapsto i, v \mapsto l, q\mapsto q^3,
a\mapsto q^{-2}, b\mapsto q^{-4}$.

We may instead arrive at a different analytic sum-side by first simplifying the expression \eqref{eqn:Cap1eq1}.
\begin{align}
&\dfrac{\left(-xq^2;q^3\right)_\infty\left(-xq^4;q^3\right)_\infty}{\left(xq^3;q^3\right)_\infty\left(x^2q^3;q^6\right)_\infty}
=\dfrac{\left(-xq^2;q^3\right)_\infty\left(-xq^4;q^3\right)_\infty}{\left(xq^3;q^3\right)_\infty\left(x^3q^3;q^6\right)_\infty}
\cdot\dfrac{\left(xq^2;q^3\right)_\infty\left(xq^4;q^3\right)_\infty}{\left(xq^2;q^3\right)_\infty\left(xq^4;q^3\right)_\infty}\nonumber\\
&=\dfrac{\left(x^2q^4;q^6\right)_\infty\left(x^2q^8;q^6\right)_\infty}
{\left(xq^2;q^3\right)_\infty\left(xq^3;q^3\right)_\infty\left(xq^4;q^3\right)_\infty\left(x^2q^3;q^6\right)_\infty}
=
\dfrac{\left(x^2q^4;q^6\right)_\infty\left(x^2q^8;q^6\right)_\infty}
{\left(xq^2;q\right)_\infty\left(x^2q^3;q^6\right)_\infty}
\cdot
\dfrac{\left(x^2q^6;q^6\right)_\infty}{\left(x^2q^6;q^6\right)_\infty}\nonumber\\
&=
\dfrac{\left(x^2q^4;q^2\right)_\infty}
{\left(xq^2;q\right)_\infty\left(x^2q^3;q^3\right)_\infty}
=\dfrac{\left(-xq^2;q\right)_\infty\left(xq^2;q\right)_\infty}
{\left(xq^2;q\right)_\infty\left(x^2q^3;q^3\right)_\infty}
=\dfrac{\left(-xq^2;q\right)_\infty}
{\left(x^2q^3;q^3\right)_\infty}
=\left(\sum\limits_{i\geq 0}\dfrac{x^iq^{2i+\frac{i(i-1)}{2}}}{(q;q)_i}\right)
\left(\sum\limits_{j\geq 0}\dfrac{x^{2j}q^{3j}}{\left(q^3;q^3\right)_j}\right)
\nonumber\\
&=\sum\limits_{i,j\geq 0}\dfrac{x^{i+2j}q^{2i+\frac{i(i-1)}{2}+3j}}{(q;q)_i\left(q^3;q^3\right)_j}
\end{align}
Now letting $x^m\mapsto x^mq^{3m(m-1)/2}$:
\begin{align}
\sum\limits_{i,j\geq 0}\dfrac{x^{i+2j}q^{\frac{3(i+2j)(i+2j-1)}{2}+2i+\frac{i(i-1)}{2}+3j}}{(q;q)_i\left(q^3;q^3\right)_j}
=\sum\limits_{i,j\geq 0}\dfrac{x^{i+2j}q^{2i^2+6ij+6j^2}}{(q;q)_i\left(q^3;q^3\right)_j}.
\end{align}

It is straightforward to repeat the above steps to analyze the second identity of Capparelli.
Everything is the same as before except for the initial blocks in $\mu$.
If $1$ appears in $\mu$ then the maximal block corresponding to $1$ and $2$ in $\mu$ is of the shape
$1,[2,1]^*,2^\bullet$. However, if $1$ does not appear in $\mu$ then $\mu$ must start with a part at least $3$.
The generating function for $\mu$ is thus:
\begin{align}
&\left(xq\dfrac{1}{1-xq^2\cdot xq}(1+xq^2) +1\right)
\left(\prod\limits_{j\geq 1} \dfrac{1}{1-xq^{3j}}\right)
\left(\prod\limits_{j\geq 1}1+xq^{3j+1}\right)
\left(\prod\limits_{j\geq 1} \dfrac{1}{1-xq^{3j+2}\cdot xq^{3j+1}}\cdot (1+xq^{3j+2})\right)
\nonumber\\
&=	\left(xq^3;q^3\right)_\infty^{-1}
\left(-xq^5;q^3\right)_\infty\left(-xq;q^3\right)_\infty\left(x^2q^3;q^6\right)_\infty^{-1}\label{eqn:Cap2eq1}\\
&=\sum\limits_{i,j,k,l\geq 0}\dfrac{x^{i+j+k+2l}
	q^{3i + 5j + \frac{3j(j-1)}{2} + k + \frac{3k(k-1)}{2}  + 3l}
}{\left(q^3;q^3\right)_i\left(q^3;q^3\right)_j\left(q^3;q^3\right)_k\left(q^6;q^6\right)_l}.
\end{align}
Putting back a $3$-staircase:
\begin{align}
\sum\limits_{i,j,k,l\geq 0}\dfrac{x^{i+j+k+2l}
	q^{\frac{3(i+j+k+2l)(i+j+k+2l-1)}{2}+3i + 5j + \frac{3j(j-1)}{2} + k + \frac{3k(k-1)}{2}  + 3l}
}{\left(q^3;q^3\right)_i\left(q^3;q^3\right)_j\left(q^3;q^3\right)_k\left(q^6;q^6\right)_l}.
\end{align}
One may again arrive at a different analytic sum-side by simplifying expression \eqref{eqn:Cap2eq1} first:
\begin{align}
&\dfrac{\left(-xq^5;q^3\right)_\infty\left(-xq;q^3\right)_\infty}{\left(xq^3;q^3\right)_\infty\left(x^2q^3;q^6\right)_\infty}
=\dfrac{\left(-xq^5;q^3\right)_\infty\left(-xq;q^3\right)_\infty}{\left(xq^3;q^3\right)_\infty\left(x^3q^3;q^6\right)_\infty}
\cdot\dfrac{\left(xq^5;q^3\right)_\infty\left(xq;q^3\right)_\infty}{\left(xq^5;q^3\right)_\infty\left(xq;q^3\right)_\infty}\nonumber\\
&=\dfrac{\left(x^2q^2;q^6\right)_\infty\left(x^2q^{10};q^6\right)_\infty}
{\left(xq;q^3\right)_\infty\left(xq^3;q^3\right)_\infty\left(xq^5;q^3\right)_\infty\left(x^2q^3;q^6\right)_\infty}
=\dfrac{(1-xq^2)}{(1-x^2q^4)}\dfrac{\left(x^2q^2;q^6\right)_\infty\left(x^2q^{4};q^6\right)_\infty}
{\left(xq;q^3\right)_\infty\left(xq^2;q^3\right)_\infty\left(xq^3;q^3\right)_\infty\left(x^2q^3;q^6\right)_\infty}
\nonumber\\
&=\dfrac{(1-xq^2)}{(1-x^2q^4)}\dfrac{\left(x^2q^2;q^6\right)_\infty\left(x^2q^{4};q^6\right)_\infty}
{\left(xq;q\right)_\infty\left(x^2q^3;q^6\right)_\infty}\cdot\dfrac{\left(x^2q^6;q^6\right)_\infty}{\left(x^2q^6;q^6\right)_\infty}
=\dfrac{(1-xq^2)}{(1-x^2q^4)}\dfrac{\left(x^2q^2;q^2\right)_\infty}
{\left(xq;q\right)_\infty\left(x^2q^3;q^3\right)_\infty}
=\dfrac{(-xq;q)_\infty}{(1+xq^2)\left(x^2q^3;q^3\right)_\infty}
\nonumber\\
&=(1+xq)\dfrac{(-xq^3;q)_\infty}{\left(x^2q^3;q^3\right)_\infty}
=(1+xq)\left(\sum\limits_{i\geq 0}\dfrac{x^iq^{3i + \frac{i(i-1)}{2}}}{(q;q)_i}\right)
\left(\sum\limits_{j\geq 0}\dfrac{x^{2j}q^{3j}}{\left(q^3;q^3\right)_j}\right)\nonumber\\
&=\sum\limits_{i,j\geq 0}\dfrac{x^{i+2j}q^{3i + \frac{i(i-1)}{2}+3j } + x^{i+2j+1}q^{3i + \frac{i(i-1)}{2}+3j+1 }}{(q;q)_i\left(q^3;q^3\right)_j}.
\end{align}
Putting back a $3$-staircase:
\begin{align}
&\sum\limits_{i,j\geq 0}\dfrac{x^{i+2j}q^{\frac{3(i+2j)(i+2j-1)}{2}+3i + \frac{i(i-1)}{2}+3j } + x^{i+2j+1}q^{
		\frac{3(i+2j+1)(i+2j)}{2}+3i + \frac{i(i-1)}{2}+3j+1 }}{(q;q)_i\left(q^3;q^3\right)_j}\nonumber\\
	&=
	\sum\limits_{i,j\geq 0}\dfrac{x^{i+2j}q^{2i^2+6ij+6j^2+i}}{(q;q)_i\left(q^3;q^3\right)_j}
	+\sum\limits_{i,j\geq 0}\dfrac{x^{i+2j+1}q^{2i^2+6ij+6j^2+4i+6j+1}}{(q;q)_i\left(q^3;q^3\right)_j}.
\end{align}
Since Capparelli's identities are true \cite{And-cap,AAG,C2,DL,MP2, TX}, we have:
\begin{thm} The following formal power series identities hold:
\begin{align}
\sum\limits_{i,j\geq 0}\dfrac{q^{2i^2+6ij+6j^2}}{(q;q)_i\left(q^3;q^3\right)_j}
&=\dfrac{1}{\left(q^2,q^3,q^9,q^{10};q^{12}\right)_\infty},\\
\sum\limits_{i,j,k,l\geq 0}\dfrac{
	q^{\frac{3(i+j+k+2l)(i+j+k+2l-1)}{2}+3i + 5j + \frac{3j(j-1)}{2} + k + \frac{3k(k-1)}{2}  + 3l}
}{\left(q^3;q^3\right)_i\left(q^3;q^3\right)_j\left(q^3;q^3\right)_k\left(q^6;q^6\right)_l}
&=
\sum\limits_{i,j\geq 0}\dfrac{q^{2i^2+6ij+6j^2+i}}{(q;q)_i\left(q^3;q^3\right)_j}
+\sum\limits_{i,j\geq 0}\dfrac{q^{2i^2+6ij+6j^2+4i+6j+1}}{(q;q)_i\left(q^3;q^3\right)_j}\nonumber\\
&=\left(-q,-q^3,-q^4,-q^6;q^6\right)_\infty.
\end{align}		
\end{thm}

\end{document}